\documentclass[pdflatex,sn-mathphys-num]{sn-jnl}

\usepackage{graphicx}%
\usepackage{multirow}%
\usepackage{amsmath,amssymb,amsfonts}%
\usepackage{amsthm}%
\usepackage{mathrsfs}%
\usepackage[title]{appendix}%
\usepackage{xcolor}%
\usepackage{textcomp}%
\usepackage{manyfoot}%
\usepackage{booktabs}%
\usepackage{algorithm}%
\usepackage{algorithmicx}%
\usepackage{algpseudocode}%
\usepackage{listings}%
\usepackage{enumerate}%

\theoremstyle{thmstyleone}%
\newtheorem{theorem}{Theorem}
\newtheorem{proposition}[theorem]{Proposition}%
\newtheorem{lemma}[theorem]{Lemma}%
\newtheorem{corollary}[theorem]{Corollary}%

\theoremstyle{thmstyletwo}%
\newtheorem{example}{Example}%
\newtheorem{remark}{Remark}%

\theoremstyle{thmstylethree}%

\newcommand{\Z}{\mathbb{Z}}
\newcommand{\F}{\mathbb{F}}
\newcommand{\ord}{\operatorname{ord}}
\newcommand{\lcm}{\operatorname{lcm}}

\begin{document}

\title[Good Integers: $(T,k)$-Subclasses and Applications to Galois Duality in Coding Theory]{Good Integers: $(T,k)$-Subclasses and Applications to Galois Duality in Coding Theory}


\author*[1]{\fnm{Somphong} \sur{Jitman}}\email{sjitman@gmail.com}

\author[1]{\fnm{Panthakan} \sur{Boonsuriyatham}}\email{panthbs@gmail.com}


\affil[1]{\orgdiv{Department of Mathematics, Faculty of Science}, \orgname{Silpakorn University},  \orgaddress{\street{Mueang Nakhon Pathom}, \city{Nakhon Pathom}, \postcode{73000},   \country{Thailand}}}




\abstract{The notion of good integers, namely the divisors of the sequence $(a^s+b^s)_{s\ge 1}$ for nonzero coprime integers $a$ and $b$, together with their subfamilies such as oddly-good and evenly-good integers, has become  an important arithmetic tool in the study of Euclidean and Hermitian dualities for abelian and cyclic codes. Building on this perspective, this paper introduces and studies another interesting subclass of good integers arising from the sequence
$
\bigl(a^{ks+T}+b^{ks+T}\bigr)_{s\ge 1}$ for some integers $0\leq T<k$,
whose divisors are called $(T,k)$-{\em good integers with respect to} $(a,b)$.
 An arithmetic theory of these integers is developed, including a characterization at odd prime powers, a general characterization for odd integers in terms of $2$-adic valuations, and a treatment of even integers. An explicit algorithm is also given for deciding whether a given integer $d$ is $(T,k)$-good with respect to $(a,b)$ and, when it is, for computing an exponent $s$ such that $d\mid \bigl(a^{ks+T}+b^{ks+T}\bigr)$.  Applications in coding theory are then obtained from the specialization $(a,b)=(q,1)$, where $q$ is a prime power. In particular, the $q^k$-cyclotomic classes of the cyclic group $\mathbb Z_n$ characterize the Galois self-reciprocal irreducible factors of $x^n-1$ over $\F_{q^k}$, give a description and enumeration of Galois LCD cyclic codes of length $n$ over $\F_{q^k}$, and lead to a characterization of Galois self-dual  cyclic codes.}

\keywords{Good integers, cyclic codes, Galois duality, Galois self-reciprocal polynomials, Galois  Linear complementary dual  codes, Galois  self-dual codes}


\pacs[MSC Classification]{11A07, 94B15, 11T71}

\maketitle

\section{Introduction} \label{sec1}
The classical notion of a good integer \cite{Moree1997} is formulated for a pair of nonzero coprime integers $a$ and $b$: a positive integer $d$ is said to be \emph{good with respect to} $a$ and $b$ if it is a divisor of some term of the sequence $(a^s+b^s)_{s\ge 1}$. Equivalently, $
d\mid (a^s+b^s)$
for some positive integer $s$.  The set of all such integers is denoted by $G(a,b)$.  In \cite{J2018}, two important subclasses of good integers were also investigated. A positive integer $d$ is said to be \emph{oddly-good with respect to} $a$ and $b$ if $
d\mid (a^s+b^s)$
for some odd positive integer $s$, and \emph{evenly-good with respect to} $a$ and $b$ if $
d\mid (a^s+b^s)$
for some even positive integer $s$.  The sets of all oddly-good and evenly-good integers with respect to $a$ and $b$ are denoted by $OG(a,b)$ and $EG(a,b)$, respectively.
These arithmetic notions have been studied from several directions, especially through multiplicative orders modulo primes and prime powers; see, for example,   \cite{J2018,Jitman2020,Moree1997}.

For applications in coding theory, the relevant specialization is the case $(a,b)=(q,1)$, where $q$ is a prime power. In this case, the problem is to determine whether $d$ is a divisor of some term of the sequence $(q^s+1)_{s\ge 1}$. This divisibility condition is used in the characterization of the reciprocal behavior of irreducible monic factors of $x^n-1$ over the finite field $\F_q$.  Consequently, it plays an important role in the study of Euclidean duals of cyclic codes~\cite{JLX2011}, negacyclic codes~\cite{Jitman2020}, and abelian codes~\cite{JLLX2013}. For example, it arises in the characterization and enumeration of important families of codes such as Euclidean self-dual codes, Euclidean linear complementary dual (LCD) codes, and Euclidean hulls of codes in  \cite{J2018,JLLX2013,JLX2011}.  The motivation becomes even stronger when one moves from Euclidean duality to Hermitian duality, where oddly-good integers arise naturally in the characterization  and enumeration of Hermitian self-dual cyclic, abelian,  and quasi-abelian codes; see, for example, \cite{BJU2018,J2018,JLS2014}.

 More generally, Galois self-dual and Galois complementary dual codes provide a common framework that contains the Euclidean and Hermitian cases as special instances. This general setting has been investigated for cyclic, constacyclic, and related classes of codes; see, for example, the work on Galois self-dual constacyclic and  quasi constacyclic codes in \cite{FanZhang2017, FL2024},   Galois linear complementary dual codes in \cite{LiuLiu2018, LLY2020}, and  two-sided Galois duals  in \cite{T2023}. 
Motivated by these developments, for fixed nonzero coprime integers $a$ and $b$, and integers $k$ and $T$ with $0\le T<k$, this paper introduces and investigates the family of positive integers $d$ satisfying
\[
d\mid \bigl(a^{ks+T}+b^{ks+T}\bigr)
\]
for some positive integer $s$. More specifically, $d$ is called a \emph{$(T,k)$-good integer with respect to $a$ and $b$}. The set of all such integers is denoted by $
G_{(T,k)}(a,b)$. From the definitions, it is clear that $
G_{(T,k)}(a,b) \subseteq G(a,b)$.
Since $\gcd(a,b)=1$, no prime divisor of $ab$ can divide a term of the form $a^{ks+T}+b^{ks+T}$. Consequently,  we have the following lemma. 
\begin{lemma}\label{lem:coprime-to-ab}
Let $a$ and $b$ be nonzero coprime integers and let $k$ and $T$ be integers such that $0\leq T<k$. Then every element of $G_{(T,k)}(a,b)$ is coprime to $ab$.
\end{lemma}
When $k=1$ and $T=0$, this reduces to the classical notion of good integers with respect to  $a$ and $b$. When $k=2$  and $T=1$, it becomes
\[
d\mid \bigl(a^{2s+1}+b^{2s+1}\bigr),
\]
which corresponds to the classical notion of oddly-good integers  with respect to  $a$ and $b$. Thus, $G_{(0,1)}(a,b)= G(a,b)$, $G_{(0,2)}(a,b)= EG(a,b)$,  and $G_{(1,2)}(a,b)= OG(a,b)$.  Hence, the family $G_{(T,k)}(a,b)$ extends the classical theory of good integers and provides an   arithmetic framework for the   study of Euclidean, Hermitian, and Galois dualities.

The purpose of this paper is twofold. First, a self-contained arithmetic theory of $G_{(T,k)}(a,b)$ is developed. This includes local criteria at odd prime powers,  a general   characterization of odd integers in terms of the $2$-adic valuations, a treatment of the even case, and an explicit algorithm for deciding whether a given integer belongs to $G_{(T,k)}(a,b)$ and,  when it does, for computing an exponent $s$ such that $
d\mid \bigl(a^{ks+T}+b^{ks+T}\bigr)$. Second, applications in coding theory are obtained by specializing to $(a,b)=(q,1)$. In this case, the arithmetic is transferred to $q^k$-cyclotomic classes on finite cyclic and abelian groups. For cyclic groups, the resulting orbit conditions determine when an irreducible factor of $x^n-1$ over $\F_{q^k}$ is $\theta$-self-reciprocal, where $\theta(c)=c^{q^T}$. This leads to a description and enumeration of Galois LCD cyclic codes over $\F_{q^k}$. It also leads to a characterization of Galois self-dual  cyclic codes, which reduces to the familiar pairwise form in the involutory case.

   The paper is organized as follows. Section~\ref{sec2} develops the arithmetic theory of $(T,k)$-good integers with respect to $(a,b)$, including local criteria at odd prime powers, global characterizations for odd and even integers, several special cases, and relations among the sets $G_{(T,k)}(a,b)$. Section~\ref{sec3} presents an explicit algorithm for deciding membership in $G_{(T,k)}(a,b)$ and for constructing a corresponding exponent, together with illustrative examples. Section~\ref{sec4}  specializes to $(a,b)=(q,1)$ and transfers the arithmetic to $q^k$-cyclotomic classes on finite abelian groups, leading to characterizations, counting formulas, and further relations among classes of type $T$.  Section~\ref{sec5} applies the cyclic group case to Galois duality of cyclic codes and presents criteria for $\theta$-self-reciprocal irreducible factors of $x^n-1$ over $\F_{q^k}$, together with a description and enumeration of Galois LCD cyclic codes. It also provides a characterization of Galois self-dual  cyclic codes, which reduces to a simpler pairwise form in the involutory case.  The paper ends with concluding remarks and possible directions for further research in Section \ref{sec6}.

\section{Characterization of $G_{(T,k)}(a,b)$}\label{sec2}
In this section, the arithmetic structure of $G_{(T,k)}(a,b)$ is developed. The defining condition
\[
d\mid \bigl(a^{ks+T}+b^{ks+T}\bigr)
\quad\text{for some } s\ge 1
\]
naturally leads to multiplicative orders of $a b^{-1}$ modulo divisors of $d$. Basic arithmetic properties required for the study of good integers are recalled.   The local problem at odd prime powers is treated first, then combined into general criteria for odd and even integers.

\subsection{Basic arithmetic properties}
For a positive integer $m$ coprime to an integer $i$, the multiplicative order of $i$ modulo $m$ is denoted by $\ord_m(i)$. For a prime $p$ and an integer $i\neq 0$, the $p$-adic valuation $\nu_p(i)$ is the largest integer $e\ge 0$ such that $p^e\mid i$. Let $a$ and $b$  be nonzero coprime integers. For each positive integer  $d$ coprime to $ab$, the inverse of $b$ modulo $d$ exists and it is denoted by $ b^{-1}$. This allows the divisibility condition $d\mid (a^s+b^s)$ to be translated into a congruence involving the residue class $a b^{-1}$ modulo $d$ in \cite{Moree1997}, thereby reducing the problem to the multiplicative order of $a b^{-1}$ modulo $d$. 

The following lemma translates the divisibility condition $d\mid (a^N+b^N)$ into a statement about the multiplicative order of $ab^{-1}$ modulo $d$. It is a basic arithmetic tool in the later study of $(T,k)$-good integers   with respect to $a$ and $b$.

\begin{lemma}\label{lem:minusone}
Let $d$ be an integer such that  $d\geq 3$ and  $\gcd(d,ab)=1$ and let $N$ be a positive integer.  Then  $
a^N+b^N\equiv 0\pmod d $
  if and only if $\ord_d(a b^{-1})$ is even and $
N\equiv \frac{\ord_d(a b^{-1})}{2}\pmod {\ord_d(a b^{-1})}$.
\end{lemma}

 \begin{proof}
 Assume that 
$
a^N+b^N\equiv 0\pmod d$.  Then  $
(ab^{-1})^N\equiv -1\pmod d$ which implies that $
(ab^{-1})^{2N}\equiv 1\pmod d$. 
Equivalently, 
$\ord_d(ab^{-1})\mid 2N$.
Since
$
(ab^{-1})^N\not\equiv 1\pmod d$,
we have 
$
\ord_d(ab^{-1})\nmid N$.
It follows that $\ord_d(ab^{-1})$ is even and
\[
N\equiv \frac{\ord_d(ab^{-1})}{2}\pmod{\ord_d(ab^{-1})}.
\]

Conversely, assume that $\ord_d(ab^{-1})$ is even and
\[
N\equiv \frac{\ord_d(ab^{-1})}{2}\pmod{\ord_d(ab^{-1})}.
\]
Then $
(ab^{-1})^N\equiv (ab^{-1})^{\ord_d(ab^{-1})/2}\pmod d$.
The residue class
$
(ab^{-1})^{\ord_d(ab^{-1})/2}$
has order $2$, and hence,
\[
(ab^{-1})^N\equiv(ab^{-1})^{\ord_d(ab^{-1})/2}\equiv -1\pmod d.
\]
Equivalently, 
$
a^N+b^N\equiv 0\pmod d$.
\end{proof}

The  characterization of odd  good integers is given  in terms of the $2$-adic valuations of the orders $\ord_p(a b^{-1})$ for primes $p\mid m$ in \cite{Moree1997}.    

\begin{theorem}[{\cite[Theorem 1]{Moree1997}}]\label{cor:minusone-odd}
Let $d>1$ be an odd integer such that $\gcd(d,ab)=1$. Then
$
d\in G(a,b)$
if and only if there exists a positive integer $\alpha$ such that
$
\nu_2\bigl(\ord_p(a b^{-1})\bigr)=\alpha $
for all primes $ p\mid d$.
\end{theorem}

\subsection{Odd prime power characterization}
In this subsection, we focus on the local problem at a single odd prime power. The corresponding prime power conditions will later be combined, via the Chinese Remainder Theorem, to obtain criteria for general odd integers in Section \ref{subsec:odd}.

\begin{proposition}[{\cite[Theorem 3.6]{N2000}}] \label{prop:liftorder}
Let $p$ be an odd prime such that  $p\nmid ab$ and  $e$ be a positive integer. Let $ b^{-1}$ denote the inverse of $b$ modulo $p^e$ and let $
\lambda_p=\nu_p\bigl((a b^{-1})^{\ord_p(a b^{-1})}-1\bigr)$.
Then 
\[
\ord_{p^e}(a b^{-1})=
\begin{cases}
\ord_p(a b^{-1}) &\text{ if } e\le \lambda_p,\\[1mm]
\ord_p(a b^{-1})p^{e-\lambda_p} & \text{ if } e>\lambda_p.
\end{cases}
\]
In particular,
$
\nu_2\bigl(\ord_{p^e}(a b^{-1})\bigr)=\nu_2\bigl(\ord_p(a b^{-1})\bigr) $ for all  $e\ge 1$.
\end{proposition}

The characterization of  a $(T,k)$-good  odd prime power with respect to $a$ and $b$ is given in the following theorem. 

\begin{theorem}\label{thm:local}
Let $p$ be an odd prime such that $p\nmid ab$, let $e$ be a positive integer, and let $0\le T<k$ be integers. Then
$
p^e\in G_{(T,k)}(a,b)$ 
if and only if $\ord_p(ab^{-1})$ is even and
\[
T\equiv \frac{\ord_{p^e}(ab^{-1})}{2}
\pmod{\gcd(k,\ord_{p^e}(ab^{-1}))}.
\]
\end{theorem}

\begin{proof}
Assume that
$
p^e\in G_{(T,k)}(a,b)$.
Then there exists a positive integer $s$ such that
$
a^{ks+T}+b^{ks+T}\equiv 0\pmod{p^e}$, which is equivalent to $
(ab^{-1})^{ks+T}\equiv -1\pmod{p^e}$.
By Lemma~\ref{lem:minusone}, it follows that $\ord_{p^e}(ab^{-1})$ is even and
\[
ks+T\equiv \frac{\ord_{p^e}(ab^{-1})}{2}\pmod{\ord_{p^e}(ab^{-1})}.
\]
Hence, the congruence
\[
ks\equiv \frac{\ord_{p^e}(ab^{-1})}{2}-T
\pmod{\ord_{p^e}(ab^{-1})}
\]
is solvable, which implies that
\[
\gcd(k,\ord_{p^e}(ab^{-1}))
\mid
\left(\frac{\ord_{p^e}(ab^{-1})}{2}-T\right).
\]
Consequently,
\[
T\equiv \frac{\ord_{p^e}(ab^{-1})}{2}
\pmod{\gcd(k,\ord_{p^e}(ab^{-1}))}.
\]
By Proposition~\ref{prop:liftorder},
$
\nu_2\bigl(\ord_{p^e}(ab^{-1})\bigr)=\nu_2\bigl(\ord_p(ab^{-1})\bigr)$, which implies that  $\ord_{p^e}(ab^{-1})$ is even if and only if $\ord_p(ab^{-1})$ is even.

Conversely, assume that $\ord_p(ab^{-1})$ is even and $
T\equiv \frac{\ord_{p^e}(ab^{-1})}{2}
\pmod{\gcd(k,\ord_{p^e}(ab^{-1}))}$.
Then $\ord_{p^e}(ab^{-1})$ is even by Proposition~\ref{prop:liftorder}. Moreover,
\[
\gcd(k,\ord_{p^e}(ab^{-1}))
\mid
\left(\frac{\ord_{p^e}(ab^{-1})}{2}-T\right),
\]
which implies that  the linear congruence
$
ks\equiv \frac{\ord_{p^e}(ab^{-1})}{2}-T
\pmod{\ord_{p^e}(ab^{-1})}$
has an integer solution. Hence, there exists a positive integer $s$ such that
\[
ks+T\equiv \frac{\ord_{p^e}(ab^{-1})}{2}\pmod{\ord_{p^e}(ab^{-1})}.
\]
Since $\ord_{p^e}(ab^{-1})$ is even, 
$
(ab^{-1})^{ks+T}\equiv -1\pmod{p^e}$ by Lemma~\ref{lem:minusone}. Multiplying by $b^{ks+T}$, we obtain $
a^{ks+T}+b^{ks+T}\equiv 0\pmod{p^e}$.
Therefore, $
p^e\in G_{(T,k)}(a,b)$ as desired.
\end{proof}

\begin{corollary}\label{cor:local-2adic}
Let $p$ be an odd prime with $p\nmid ab$, let $e$ be a positive integer. 
Then $
p^e\in G_{(T,k)}(a,b)$ 
if and only if there exists a positive integer $\alpha$ satisfying the  following conditions:
\begin{enumerate}[$(1)$]

\item 
$
\nu_2\bigl(\ord_p(ab^{-1})\bigr)=\alpha$.
\item
$
T\equiv 2^{\alpha-1}\pmod{2^{\min\{\nu_2(k),\alpha\}}},
$
\item $
\gcd\left(\frac{k}{2^{\nu_2(k)}},
\frac{\ord_{p^e}(ab^{-1})}{2^\alpha}\right)\mid T.
$
\end{enumerate}
\end{corollary}
\begin{proof}
Assume  that
$
p^e\in G_{(T,k)}(a,b)$.
By Theorem~\ref{thm:local}, $\ord_p(ab^{-1})$ is even and
\[
T\equiv \frac{\ord_{p^e}(ab^{-1})}{2}
\pmod{\gcd(k,\ord_{p^e}(ab^{-1}))}.
\]
Let $\alpha=\nu_2(\ord_p(ab^{-1}))  $. Then $\alpha$ is positive and satisfies (1). 
 Since
$
\nu_2(\ord_{p^e}(ab^{-1}))=\nu_2(\ord_p(ab^{-1}))=\alpha$, 
we have 
\[
\gcd(k,\ord_{p^e}(ab^{-1}))
=
2^{\min\{\nu_2(k),\alpha\}}
\gcd\left(\frac{k}{2^{\nu_2(k)}},
\frac{\ord_{p^e}(ab^{-1})}{2^\alpha}\right).
\]
This implies
\[
T\equiv \frac{\ord_{p^e}(ab^{-1})}{2}
\equiv 2^{\alpha-1}
\pmod{2^{\min\{\nu_2(k),\alpha\}}},
\]
which is (2), and also
\[
\gcd\left(\frac{k}{2^{\nu_2(k)}},
\frac{\ord_{p^e}(ab^{-1})}{2^\alpha}\right)\mid T,
\]
which is (3).

Conversely, assume that there exists a positive integer $\alpha$ satisfying  (1)--(3). Then $\ord_p(ab^{-1})$ is even, which implies that  $\ord_{p^e}(ab^{-1})$ is also even.  Since
\[
\gcd\left(\frac{k}{2^{\nu_2(k)}},
\frac{\ord_{p^e}(ab^{-1})}{2^\alpha}\right)
\mid
\frac{\ord_{p^e}(ab^{-1})}{2},
\]
condition \textup{(3)} gives
\[
T\equiv \frac{\ord_{p^e}(ab^{-1})}{2}
\pmod{
\gcd\left(\frac{k}{2^{\nu_2(k)}},
\frac{\ord_{p^e}(ab^{-1})}{2^\alpha}\right)}.
\]
Combining this with \textup{(2)} and using the coprimality of the two moduli, we obtain
\[
T\equiv \frac{\ord_{p^e}(ab^{-1})}{2}
\pmod{\gcd(k,\ord_{p^e}(ab^{-1}))}.
\]
Hence, by Theorem~\ref{thm:local}, $p^e\in G_{(T,k)}(a,b)$.
\end{proof}

\subsection{Odd integers:  $2$-adic characterization} \label{subsec:odd}

We now turn to arbitrary odd integers. In this setting, the local conditions at the prime power divisors of $d$ combine into a particularly   global criterion  expressed in terms of the $2$-adic valuations of the orders $\ord_p(ab^{-1})$ for all primes $p\mid d$.

For a positive integer $d$ with distinct prime power factorization
\begin{align}\label{fac-d}
d=\displaystyle\prod_{i=1}^m p_i^{e_i},
\end{align}
we write
\[
\lcm_{p^e\parallel d}\ord_{p^e}(ab^{-1})
=
\lcm\bigl(\ord_{p_1^{e_1}}(ab^{-1}),\dots,\ord_{p_m^{e_m}}(ab^{-1})\bigr).
\] Then we have the following results. 

\begin{lemma}\label{lem:common-alpha}
Let $d$ be an odd integer such that $\gcd(d,ab)=1$ decomposed as in \eqref{fac-d}. If there exists an integer $N$ such that
\[
N\equiv \frac{\ord_{p_j^{e_j}}(ab^{-1})}{2}
\pmod{\ord_{p_j^{e_j}}(ab^{-1})}
\]
for all  $j\in\{1,\dots,m\}$, then
\[
\nu_2\bigl(\ord_{p_1}(ab^{-1})\bigr)
=
\nu_2\bigl(\ord_{p_2}(ab^{-1})\bigr)
=
\cdots
=
\nu_2\bigl(\ord_{p_m}(ab^{-1})\bigr).
\]
\end{lemma}

\begin{proof}
For each $j\in\{1,\dots,m\}$, write
\[
\ord_{p_j^{e_j}}(ab^{-1})=2^{\alpha_j}w_j,
\]
where $w_j$ is odd. By Proposition~\ref{prop:liftorder},
\[
\alpha_j=\nu_2\bigl(\ord_{p_j}(ab^{-1})\bigr).
\]
Assume that $\alpha_i<\alpha_j$ for some $i\ne j$. Then
\[
N\equiv \frac{\ord_{p_i^{e_i}}(ab^{-1})}{2}
=2^{\alpha_i-1}w_i
\pmod{2^{\alpha_i}}.
\]
Since $w_i$ is odd,
\[
N\equiv 2^{\alpha_i-1}\pmod{2^{\alpha_i}}.
\]
On the other hand,
\[
N\equiv \frac{\ord_{p_j^{e_j}}(ab^{-1})}{2}
=2^{\alpha_j-1}w_j
\pmod{2^{\alpha_j}}.
\]
Since $\alpha_i<\alpha_j$, this implies
\[
N\equiv 0\pmod{2^{\alpha_i}},
\]
a contradiction. Hence, all $\alpha_j$ are equal.
\end{proof}

\begin{lemma}\label{lem:global-compatibility}
Let $d$ be an odd integer such that $\gcd(d,ab)=1$ decomposed as in \eqref{fac-d}. Assume that, for each $j\in\{1,\dots,m\}$,
$
\ord_{p_j^{e_j}}(ab^{-1})=2^\alpha w_j$,
where $\alpha$ is a positive integer and each $w_j$ is odd. If
$
T\equiv 2^{\alpha-1}\pmod{2^{\min\{\alpha,\nu_2(k)\}}}
$
and
$
\gcd\left(\frac{k}{2^{\nu_2(k)}},\lcm(w_1,\dots,w_m)\right)\mid T$,
then the system
\begin{align*}
N&\equiv T\pmod{k},\\
N&\equiv \frac{\ord_{p_j^{e_j}}(ab^{-1})}{2}
\pmod{\ord_{p_j^{e_j}}(ab^{-1})}
\quad (j=1,\dots,m),
\end{align*}
is solvable.
\end{lemma}

\begin{proof}
Assume that $
T\equiv 2^{\alpha-1}\pmod{2^{\min\{\alpha,\nu_2(k)\}}}
$
and
$
\gcd\left(\frac{k}{2^{\nu_2(k)}},\lcm(w_1,\dots,w_m)\right)\mid T$. 

First, we  show that the  congruences
\[
N\equiv \frac{\ord_{p_j^{e_j}}(ab^{-1})}{2}
\pmod{\ord_{p_j^{e_j}}(ab^{-1})}
\quad (j=1,\dots,m)
\]
are pairwise compatible. Let $i,j\in\{1,\dots,m\}$. Then
\[
\frac{\ord_{p_i^{e_i}}(ab^{-1})}{2}
-
\frac{\ord_{p_j^{e_j}}(ab^{-1})}{2}
=
2^{\alpha-1}(w_i-w_j).
\]
Since $w_i$ and $w_j$ are odd, the difference $w_i-w_j$ is even  and  $
\gcd(w_i,w_j)\mid (w_i-w_j)$.
Hence, 
\[
2^\alpha\gcd(w_i,w_j)
\mid
\left(
\frac{\ord_{p_i^{e_i}}(ab^{-1})}{2}
-
\frac{\ord_{p_j^{e_j}}(ab^{-1})}{2}
\right).
\]
Since 
$
\gcd\bigl(\ord_{p_i^{e_i}}(ab^{-1}),\ord_{p_j^{e_j}}(ab^{-1})\bigr)
=
2^\alpha\gcd(w_i,w_j)$,
it follows that the congruences are pairwise compatible.

Next, we show that
$
N\equiv T\pmod{k}$ 
is compatible with each local congruence. Let $j$ be fixed in  $\{1,\dots,m\}$. Since
\[
\gcd\left(\frac{k}{2^{\nu_2(k)}},w_j\right)
\mid
\gcd\left(\frac{k}{2^{\nu_2(k)}},\lcm(w_1,\dots,w_m)\right),
\]
the hypothesis implies that
\[
\gcd\left(\frac{k}{2^{\nu_2(k)}},w_j\right)\mid T.
\]
We note that 
$
\gcd\bigl(k,\ord_{p_j^{e_j}}(ab^{-1})\bigr)
=
2^{\min\{\alpha,\nu_2(k)\}}
\gcd\left(\frac{k}{2^{\nu_2(k)}},w_j\right)$ 
and the two factors on the right hand side are coprime. Since
\[
\frac{\ord_{p_j^{e_j}}(ab^{-1})}{2}=2^{\alpha-1}w_j
\]
with $w_j$ odd, the congruence $
T\equiv 2^{\alpha-1}\pmod{2^{\min\{\alpha,\nu_2(k)\}}}$
implies that 
\[
T\equiv \frac{\ord_{p_j^{e_j}}(ab^{-1})}{2}
\pmod{2^{\min\{\alpha,\nu_2(k)\}}}.
\]
Since 
$
\gcd\left(\frac{k}{2^{\nu_2(k)}},w_j\right)
\mid
\frac{\ord_{p_j^{e_j}}(ab^{-1})}{2}$,
the divisibility
\[
\gcd\left(\frac{k}{2^{\nu_2(k)}},w_j\right)\mid T
\]
is equivalent to
\[
T\equiv \frac{\ord_{p_j^{e_j}}(ab^{-1})}{2}
\pmod{\gcd\left(\frac{k}{2^{\nu_2(k)}},w_j\right)}.
\]
Therefore,
\[
T\equiv \frac{\ord_{p_j^{e_j}}(ab^{-1})}{2}
\pmod{\gcd\bigl(k,\ord_{p_j^{e_j}}(ab^{-1})\bigr)}.
\]
Thus, the congruence
$
N\equiv T\pmod{k}$
is compatible with the $j$th local congruence.

Hence,  all congruences in the system are pairwise compatible. By the Chinese Remainder Theorem, the system is solvable.
\end{proof}

\begin{theorem}\label{thm:2adic}
Let $d>1$ be an odd integer such that $\gcd(d,ab)=1$ decomposed as in \eqref{fac-d}. Then
$
d\in G_{(T,k)}(a,b)
$
if and only if there exists a positive integer $\alpha$ satisfying the following conditions:
\begin{enumerate}[$(1)$]
\item
$
\nu_2\bigl(\ord_p(ab^{-1})\bigr)=\alpha$
for every prime $ p\mid d$. 

\item $
T\equiv 2^{\alpha-1}\pmod{2^{\min\{\alpha,\nu_2(k)\}}}$

\item
$
\gcd\left(\frac{k}{2^{\nu_2(k)}},
\frac{\lcm_{p^e\parallel d}\ord_{p^e}(ab^{-1})}{2^\alpha}\right)\mid T$.
\end{enumerate}
\end{theorem}

\begin{proof}
Assume first that $
d\in G_{(T,k)}(a,b)$.
Then there exists a positive integer $s$ such that $
d\mid \bigl(a^{ks+T}+b^{ks+T}\bigr)$.
Hence, 
$
p_j^{e_j}\mid \bigl(a^{ks+T}+b^{ks+T}\bigr)$ 
for every $j\in \{1,2,\dots,m\}$.
By Lemma~\ref{lem:minusone}, we have 
\[
ks+T\equiv \frac{\ord_{p_j^{e_j}}(ab^{-1})}{2}
\pmod{\ord_{p_j^{e_j}}(ab^{-1})}
\] for all   $j\in \{1,2,\dots,m\}$.
Hence, by Lemma~\ref{lem:common-alpha}, condition \textup{(1)} holds for some positive integer $\alpha$.

For each  $j\in \{1,2,\dots,m\}$, write
$
\ord_{p_j^{e_j}}(ab^{-1})=2^\alpha w_j$ with odd  $w_j$. Since
\[
ks+T\equiv T\pmod{k}
\quad\text{and}\quad
ks+T\equiv \frac{\ord_{p_j^{e_j}}(ab^{-1})}{2}
\pmod{\ord_{p_j^{e_j}}(ab^{-1})},
\]
we obtain
\[
T\equiv \frac{\ord_{p_j^{e_j}}(ab^{-1})}{2}
\pmod{\gcd\bigl(k,\ord_{p_j^{e_j}}(ab^{-1})\bigr)} 
\]
for every $j\in \{1,2,\dots,m\}$. 
It follows that 
\[
\gcd\bigl(k,\ord_{p_j^{e_j}}(ab^{-1})\bigr)
=
2^{\min\{\alpha,\nu_2(k)\}}
\gcd\left(\frac{k}{2^{\nu_2(k)}},w_j\right),
\]
with coprime factors. Therefore, for every $j\in \{1,2,\dots,m\}$,
\[
T\equiv 2^{\alpha-1}\pmod{2^{\min\{\alpha,\nu_2(k)\}}}
\quad \text{and} \quad
\gcd\left(\frac{k}{2^{\nu_2(k)}},w_j\right)\mid T.
\]
The first is exactly condition \textup{(2)}. Since
\[
\lcm(w_1,\dots,w_m)
=
\frac{\lcm_{p^e\parallel d}\ord_{p^e}(ab^{-1})}{2^\alpha},
\]
the second family of divisibility conditions is equivalent to \textup{(3)}.

Conversely, assume that \textup{(1)}--\textup{(3)} hold. For each $j\in \{1,2,\dots,m\}$, define
$
w_j=\frac{\ord_{p_j^{e_j}}(ab^{-1})}{2^\alpha}$.
Then each $w_j$ is odd  and  the system
\begin{align} \label{eq-solv}
N\equiv T\pmod{k},
\quad
N\equiv \frac{\ord_{p_j^{e_j}}(ab^{-1})}{2}
\pmod{\ord_{p_j^{e_j}}(ab^{-1})}
\quad (j=1,\dots,m)
\end{align}
is solvable by Lemma~\ref{lem:global-compatibility}. Let $M$ be a solution of \eqref{eq-solv}. Since $M\equiv T\pmod{k}$, we may write
$
M=ks+T$ 
for some integer $s$.  Replacing $M$ by a sufficiently large positive integer in the same congruence class modulo the least common multiple of the moduli, we may assume that
$
M=ks+T$
with $s\ge 1$. Then
\[
ks+T\equiv \frac{\ord_{p_j^{e_j}}(ab^{-1})}{2}
\pmod{\ord_{p_j^{e_j}}(ab^{-1})}
\quad\text{for every } j=1,\dots,m.
\]
By Lemma~\ref{lem:minusone}, we have 
\[
p_j^{e_j}\mid \bigl(a^{ks+T}+b^{ks+T}\bigr)
\] for all  $j\in \{1,2,\dots,m\}$. 
Since the prime powers $p_j^{e_j}$ are pairwise coprime, it follows that
\[
d\mid \bigl(a^{ks+T}+b^{ks+T}\bigr).
\]
Hence,  
$
d\in G_{(T,k)}(a,b)$. 
This completes the proof.
\end{proof}

\subsection{The even case}
We now turn to even integers. If exactly one of $a$ and $b$ is even, then
$
G_{(T,k)}(a,b)$
contains no even integers by Lemma \ref{lem:coprime-to-ab}.  Throughout this subsection,  $a$ and $b$ are  assumed  to be odd coprime integers. 

Write $
d=2^\varepsilon d_0$ with $\varepsilon\ge 1$ and  $d_0$ is  odd. Here, the odd part is still controlled by the preceding theorem, while the additional restriction comes from the $2$-adic valuation of $a^{ks+T}+b^{ks+T}$.

 \begin{proposition}\label{prop:two-power}
Let $a$ and $b$ be odd nonzero coprime  integers. Then the following statements hold:
\begin{enumerate}[$(1)$]
    
\item 
$
2\in G_{(T,k)}(a,b)$
for all integers $k\ge 1$ and $0\le T<k$.

\item For every integer $\varepsilon\ge 2$,
$
2^\varepsilon\in G_{(T,k)}(a,b)$
if and only if $
\varepsilon\le \nu_2(a+b)$ and there exists an integer $s\ge 1$ such that
$
ks+T $ is odd.
\end{enumerate}
\end{proposition}

\begin{proof}   Since $a$ and $b$ are odd, it follows that, for every positive integer $N$,
\begin{align} \label{eq:v2}
\nu_2\bigl(a^{N}+b^{N}\bigr)=
\begin{cases}
\nu_2(a+b) & \text{if } N \text{ is odd},\\[1mm]
1 & \text{if } N \text{ is even}.
\end{cases}
\end{align}
Consequently, $a^{ks+T}+b^{ks+T}$ is  even for all positive integers $s$, which implies that 
$
2\mid (a^{ks+T}+b^{ks+T})$.   Therefore, 
$
2\in G_{(T,k)}(a,b)$ which proves (1).

To prove (2), assume that $\varepsilon\ge 2$. By the definition,
$
2^\varepsilon\in G_{(T,k)}(a,b)
$
if and only if there exists an integer $s\ge 1$ such that
$
2^\varepsilon\mid \bigl(a^{ks+T}+b^{ks+T}\bigr)$. Assume that
$
2^\varepsilon\in G_{(T,k)}(a,b)$. Then there exists a positive  integer $s$ such that $
2^\varepsilon\mid \bigl(a^{ks+T}+b^{ks+T}\bigr)$,
and hence,
$
\varepsilon\le \nu_2\bigl(a^{ks+T}+b^{ks+T}\bigr)$.
Since $\varepsilon\ge 2$,   $ks+T$ must be odd by \eqref{eq:v2}. It follows that
\[
\varepsilon\le  \nu_2\bigl(a^{ks+T}+b^{ks+T}\bigr)=\nu_2(a+b).\]

Conversely, assume that  $\varepsilon\le \nu_2(a+b)$ and there exists a positive integer $s$ such that
$
ks+T$  is odd.
Since $ks+T$ is odd,  it follows that 
\[
\nu_2\bigl(a^{ks+T}+b^{ks+T}\bigr)=\nu_2(a+b).
\]
Hence,  we have 
$
\varepsilon\le \nu_2\bigl(a^{ks+T}+b^{ks+T}\bigr)$,
which implies that 
$
2^\varepsilon\mid \bigl(a^{ks+T}+b^{ks+T}\bigr)$.
Therefore, $
2^\varepsilon\in G_{(T,k)}(a,b)$ as desired.
\end{proof}

\begin{proposition}\label{prop:even-case}
Let $a$ and $b$ be odd nonzero integers and let $d_0$ be an odd positive integer. Then the following statements hold.
\begin{enumerate}[$(1)$]
\item
$
2d_0\in G_{(T,k)}(a,b)$ 
if and only if
$
d_0\in G_{(T,k)}(a,b)$.

\item For every integer $\varepsilon\ge 2$,
$
2^\varepsilon d_0\in G_{(T,k)}(a,b)$
if and only if  $\varepsilon\le \nu_2(a+b)$  and there exists a positive  integer $s$ such that
$
d_0\mid \bigl(a^{ks+T}+b^{ks+T}\bigr)$ and $
ks+T$  is odd.
\end{enumerate}
\end{proposition}

\begin{proof}
 By the  definition,
$
2d_0\in G_{(T,k)}(a,b)
$
if and only if there exists a positive  integer $s$ such that
$
2d_0\mid \bigl(a^{ks+T}+b^{ks+T}\bigr)$.
Since $a$ and $b$ are odd,
$
2\mid \bigl(a^{ks+T}+b^{ks+T}\bigr)$
for every integer $s\ge 1$. Therefore, $
2d_0\mid \bigl(a^{ks+T}+b^{ks+T}\bigr)$
if and only if
$
d_0\mid \bigl(a^{ks+T}+b^{ks+T}\bigr)$. 
This proves (1).

For \textup{(2)}, let $\varepsilon\ge 2$ be an integer. Assume that
$
2^\varepsilon d_0\in G_{(T,k)}(a,b)$.
Then there exists a positive  integer $s$ such that
$
2^\varepsilon d_0\mid \bigl(a^{ks+T}+b^{ks+T}\bigr)$.
Since $
\gcd(2^\varepsilon,d_0)=1$,
it follows that $
2^\varepsilon\mid \bigl(a^{ks+T}+b^{ks+T}\bigr)$ and $ 
d_0\mid \bigl(a^{ks+T}+b^{ks+T}\bigr)$.
By Proposition~\ref{prop:two-power}, the first divisibility implies that $\varepsilon\le \nu_2(a+b)$ and $
ks+T$ is odd.

Conversely, assume that $\varepsilon\le \nu_2(a+b)$  and there exists a positive integer $s$ such that $
d_0\mid \bigl(a^{ks+T}+b^{ks+T}\bigr)$ and $
ks+T$  is odd. 
Since $\varepsilon\ge 2$ and $ks+T$ is odd, we have   $
2^\varepsilon\mid \bigl(a^{ks+T}+b^{ks+T}\bigr)$ by Proposition~\ref{prop:two-power}. 
Together with $
d_0\mid \bigl(a^{ks+T}+b^{ks+T}\bigr)$ 
and $
\gcd(2^\varepsilon,d_0)=1$,
we obtain
$
2^\varepsilon d_0\mid \bigl(a^{ks+T}+b^{ks+T}\bigr)$.
Therefore, 
$
2^\varepsilon d_0\in G_{(T,k)}(a,b)$. 
This proves \textup{(2)}.
\end{proof}

\subsection{Special cases}

In this subsection, several useful refinements of Theorem~\ref{thm:2adic} are presented in the cases where $k$ is odd and where $k$ is a power of $2$. Precisely, the $2$-adic congruence condition imposes no restriction when $k$ is odd, and  the odd-part divisibility condition disappears if $k$ is a power of $2$.  

\subsubsection{The case where $k$ is odd}

When $k$ is odd, the congruence condition in Theorem~\ref{thm:2adic} disappears. Thus, membership in $G_{(T,k)}(a,b)$ for odd integers is  determined  by a common positive $2$-adic valuation of the local orders and an odd part divisibility condition on $T$.

\begin{corollary}\label{thm:k-odd}
Let $k$ be an odd positive integer and let $d>1$ be an odd integer such that  $\gcd(d,ab)=1$. Then $
d\in G_{(T,k)}(a,b)$
if and only if there exists a positive  integer $\alpha\ge 1$ such that
$
\nu_2\bigl(\ord_p(ab^{-1})\bigr)=\alpha$ for all primes $p\mid d$,
and $
\gcd\left(k,\frac{\lcm_{p^e\parallel d}\ord_{p^e}(ab^{-1})}{2^\alpha}\right)\mid T$.
\end{corollary}

\begin{proof}
Since $k$ is odd, we have  $\nu_2(k)=0$.    The congruence condition  $T\equiv 2^{\alpha-1}\pmod{2^{\min\{\alpha,\nu_2(k)\}}}$  in Theorem~\ref{thm:2adic} becomes
 $
2^{\min\{\alpha,0\}}=1$ which is always true. The result follows immediately.
\end{proof}

Under the additional assumption $\gcd(T,k)=1$, the divisibility condition in Corollary~\ref{thm:k-odd} is equivalent to
$
\gcd\left(k,\frac{\lcm_{p^e\parallel d}\ord_{p^e}(ab^{-1})}{2^\alpha}\right)=1$.
Thus, the next corollary  follows immediately  from Corollary~\ref{thm:k-odd}.

\begin{corollary}\label{cor:k-odd-coprime}
Let $k$ be an odd positive integer and let $d>1$ be an odd  integer such that  $\gcd(d,ab)=1$. If
$
\gcd(T,k)=1$,
then
$
d\in G_{(T,k)}(a,b)$
if and only if there exists a positive  integer $\alpha\ge 1$ such that
$
\nu_2\bigl(\ord_p(ab^{-1})\bigr)=\alpha$ for all primes $p\mid d$,
and
$
\gcd\left(k,\frac{\lcm_{p^e\parallel d}\ord_{p^e}(ab^{-1})}{2^\alpha}\right)=1
$.
\end{corollary}

\subsubsection{The case where $k$ is a power of $2$}

We now turn to the complementary situation where $
k=2^r$ for some positive integer $r$.
In this case, the odd-part divisibility condition in Theorem~\ref{thm:2adic} disappears completely, and the odd part of membership is  determined by the common $2$-adic valuation of the orders $\ord_p(ab^{-1})$ for primes dividing $d$.

\begin{corollary}\label{thm:k-2power}
Let $r$ be a positive integer and let $d>1$ be an odd integer such that  $\gcd(d,ab)=1$. Then $
d\in G_{(T,2^r)}(a,b)$ 
if and only if there exists a positive integer $\alpha$ such that $
\nu_2\bigl(\ord_p(ab^{-1})\bigr)=\alpha$ 
for all  primes $ p\mid d$,
and
$
T\equiv 2^{\alpha-1}\pmod{2^{\min\{r,\alpha\}}}$.
\end{corollary}

\begin{proof}
Since $
\frac{2^r}{2^{\nu_2(2^r)}}=1$, the result  follows immediately from Theorem~\ref{thm:2adic}.
\end{proof}

\begin{corollary}\label{cor:k-2power}
Let $r$ be a positive integer and let $d>1$ be an odd integer such that  $\gcd(d,ab)=1$. Then the following statements  hold.
\begin{enumerate}
\item If $T=0$, then
$
d\in G_{(0,2^r)}(a,b)$
if and only if there exists an integer $\alpha\ge r+1$ such that
$
\nu_2\bigl(\ord_p(ab^{-1})\bigr)=\alpha$
for all primes $p\mid d$.

\item If $1\le T<2^r$ and $\nu_2(T)=t$, then
$
d\in G_{(T,2^r)}(a,b)
$
if and only if
$
\nu_2\bigl(\ord_p(ab^{-1})\bigr)=t+1$ 
for all primes $p\mid d$.
\end{enumerate}
\end{corollary}

\begin{proof}
If $T=0$, then the congruence in Corollary~\ref{thm:k-2power} becomes
$
0\equiv 2^{\alpha-1}\pmod{2^{\min\{r,\alpha\}}}$,
which holds if and only if $\min\{r,\alpha\}\le \alpha-1$.  Equivalently, $\alpha\ge r+1$. This proves \textup{(1)}.

Now assume that $1\le T<2^r$ and write
\[
\nu_2(T)=t,
\]
where 
$
T=2^t u$ 
for some odd integer $u$ and  $0\le t<r$. By Corollary~\ref{thm:k-2power},
$
T\equiv 2^{\alpha-1}\pmod{2^{\min\{r,\alpha\}}}$.
Since $T$ has exact $2$-adic valuation $t$, this congruence holds if and only if
$
\alpha-1=t$. 
Hence, necessarily $
\alpha=t+1$.

Conversely, if $\alpha=t+1$, then
$
2^{\alpha-1}=2^t$. 
Since $T=2^t u$ with $u$ odd, we have
$
T\equiv 2^t \pmod{2^{t+1}}$. Since
$
\min\{r,\alpha\}=t+1$,
the required congruence follows. This proves \textup{(2)}.
\end{proof}

\subsection{Relations among sets of $(T,k)$-good integers}

In this subsection, we study the relations among sets of odd $(T,k)$-good integers as the parameter $T$ varies. For odd integers $d$, Theorem~\ref{thm:local} shows that the admissible values of $T$ are determined by congruence conditions modulo suitable divisors of $k$. This leads naturally to a family of subsets of $\Z_k$ that control how membership in $G_{(T,k)}(a,b)$ behaves  as $T$ varies.

For each integer $T$ with $0\le T<k$, define
\[
C_k(T):=\{(2i+1)T \bmod k : i\ge 0\}\subseteq \Z_k.
\]
The set $C_k(T)$ collects the residue classes modulo $k$ obtained from odd multiples of $T$. It plays a key role in the study of relations among the sets of odd $(T,k)$-good integers as the parameter $T$ varies. The following result describes $C_k(T)$ explicitly.

\begin{proposition}\label{prop:CkT}
Let $k$ and $T$ be   integers  such that $0\le T<k$. Then the following statements hold.
\begin{enumerate}[$(1)$]
\item If
$
\frac{k}{\gcd(T,k)}
$
is odd, then
\[
C_k(T)=\left\{iT \bmod k : 0\le i\le \frac{k}{\gcd(T,k)}-1\right\}
\quad \text{and }
|C_k(T)|=\frac{k}{\gcd(T,k)}.
\]

\item If
$
\frac{k}{\gcd(T,k)}
$
is even, then
\[
C_k(T)=\left\{(2i+1)T \bmod k : 0\le i\le \frac{k}{2\gcd(T,k)}-1\right\}
\quad \text{and }
|C_k(T)|=\frac{k}{2\gcd(T,k)}.
\]
\end{enumerate}
In particular,
$
|C_k(T)|=\frac{k}{\gcd(k,2\gcd(T,k))}$.
\end{proposition}

\begin{proof}  Let $\ord(T)$ denote the additive order of $T$ in  $\Z_k$. Then  
$
\ord(T)=\frac{k}{\gcd(T,k)}$.

Assume that $\ord(T)$ is odd. Since $\gcd(2,\ord(T))=1$, multiplication by $2$ is a permutation of $\Z_{\ord(T)}$. Hence, for every integer $i$, there exists an integer $j$ such that
$
i\equiv 2j+1 \pmod{\ord(T)}$.
Since $\ord(T)$ is the additive order of $T$ in $\Z_k$, it follows that
$
iT\equiv (2j+1)T \pmod{k}$.
  Therefore,  every multiple $iT$ belongs to $C_k(T)$ which implies that 
\[
C_k(T)=\{iT \bmod k : 0\le i\le \ord(T)-1\}\quad \text{and }
|C_k(T)|=\ord(T).
\]

Assume next that $\ord(T)$ is even. By definition, $
\{(2i+1)T \bmod k : 0\le i\le \tfrac{\ord(T)}{2}-1\}\subseteq C_k(T)$.
This set has exactly $\ord(T)/2$ distinct elements. Conversely, let $\ell\in C_k(T)$. Then
$
\ell\equiv (2j+1)T \pmod k$
for some integer $j\ge 0$. Since $\ord(T)$ is even, there exists an integer $j'$ with
$
0\le j'<\frac{\ord(T)}{2}$
such that
$
2j\equiv 2j' \pmod{\ord(T)}$.
Hence,
$
\ell\equiv (2j'+1)T \pmod k$,
which proves the required description of $C_k(T)$ and its cardinality.

Finally, we have 
\[
|C_k(T)|=\frac{\ord(T)}{\gcd(\ord(T),2)}
=
\frac{k}{\gcd(T,k)\,\gcd\!\left(\frac{k}{\gcd(T,k)},2\right)}
=
\frac{k}{\gcd(k,2\gcd(T,k))}
\] as desired.
\end{proof}

The next lemma gives a basic symmetry of the sets $C_k(T)$.

\begin{lemma}\label{lem:CkT-symmetry}
Let $k$ and $T$ be   integers  such that $0\le T<k$. Then
$
C_k(T)=C_k(k-T)$.
\end{lemma}

\begin{proof}
Since $k-T\equiv -T\pmod k$, we have
\[
C_k(k-T)=\{(2i+1)(k-T)\bmod k : i\ge 0\}
       =\{-(2i+1)T\bmod k : i\ge 0\}.
\]
As $i$ runs over the nonnegative integers, the odd multiples of $-T$ modulo $k$ coincide with the odd multiples of $T$ modulo $k$. Therefore
$
C_k(k-T)=C_k(T)$.
\end{proof}

The next result gives the corresponding arithmetic propagation for the sets $G_{(T,k)}(a,b)$.

\begin{proposition}\label{prop:CkT-good}
Let $k$ and $T$ be   integers  such that $0\le T<k$ and let $d>1$ be an odd integer such that $\gcd(d,ab)=1$. Then
$
d\in G_{(T,k)}(a,b)
$
if and only if
$
d\in G_{(i,k)}(a,b)
 $ for every $ i\in C_k(T)$.
\end{proposition}

\begin{proof}
Assume that
$
d\in G_{(T,k)}(a,b)$.
Let $p^e\parallel d$. By Theorem~\ref{thm:local},
$
\ord_{p^e}(ab^{-1})$  is even
and
\[
T\equiv \frac{\ord_{p^e}(ab^{-1})}{2}
\pmod{\gcd\bigl(k,\ord_{p^e}(ab^{-1})\bigr)}.
\]
Let $i\in C_k(T)$. Then $
i\equiv (2r+1)T \pmod k
$
for some integer $r\ge 0$. Reducing modulo $\gcd\bigl(k,\ord_{p^e}(ab^{-1})\bigr)$, we obtain
\[
i\equiv (2r+1)\frac{\ord_{p^e}(ab^{-1})}{2}
= \frac{\ord_{p^e}(ab^{-1})}{2}+r\,\ord_{p^e}(ab^{-1})
\equiv \frac{\ord_{p^e}(ab^{-1})}{2}
\pmod{\gcd\bigl(k,\ord_{p^e}(ab^{-1})\bigr)}.
\]
Therefore, by Theorem~\ref{thm:local},
$
p^e\in G_{(i,k)}(a,b)$
for every $ p^e\parallel d$.
Hence,
$
d\in G_{(i,k)}(a,b)$ for every $ i\in C_k(T)$.

Conversely,  the reverse implication is immediate since $T\in C_k(T)$.
\end{proof}

For convenience, let
\[
G^O_{(T,k)}(a,b)=\{d\in G_{(T,k)}(a,b): d \text{ is odd}\}
\]
denote the set of odd integers in $G_{(T,k)}(a,b)$. 
The following results describe how these sets are related as $T$ varies.  

\begin{corollary}\label{cor:type-k-minus-T-good} Let $k$ and $T$ be   integers  such that $0\le T<k$. Then
$
G^O_{(T,k)}(a,b)=G^O_{(k-T,k)}(a,b)$.
\end{corollary}

\begin{proof} 
By Proposition~\ref{prop:CkT-good}, if
$
d\in G_{(T,k)}(a,b)$,
then
$
d\in G_{(i,k)}(a,b)$
for every $ i\in C_k(T)$.
By Lemma~\ref{lem:CkT-symmetry},  we have  $k-T\in C_k(T)$ which  implies  that $
d\in G_{(k-T,k)}(a,b)$.
The converse is proved in the same way.
Therefore, $
G^O_{(T,k)}(a,b)=G^O_{(k-T,k)}(a,b)$.
\end{proof}

\begin{corollary}\label{cor:k-odd-type0-good} Let $k$ and $T$ be   integers  such that $0\le T<k$.
If  $k$ is odd, then
$
G^O_{(T,k)}(a,b)\subseteq G^O_{(0,k)}(a,b)$.
\end{corollary}

\begin{proof}
Since $\ord(T)$ divides $k$, it follows that $\ord(T)$ is odd. Hence, by Proposition~\ref{prop:CkT}, we have  $
0\in C_k(T)$.
The result now follows from Proposition~\ref{prop:CkT-good}.
\end{proof}

\begin{theorem}\label{thm:prime-k-relations}
Let $\ell$ be an odd prime. Then
\[
G^O_{(1,\ell)}(a,b)=G^O_{(2,\ell)}(a,b)=\cdots=G^O_{(\ell-1,\ell)}(a,b)\subseteq G^O_{(0,\ell)}(a,b).
\]
\end{theorem}

\begin{proof}
Let $1\le T\le \ell-1$. Since $\ell$ is prime, the additive order of $T$ in $\Z_\ell$ is equal to $\ell$ which is odd. By Proposition~\ref{prop:CkT},   we have  $
C_\ell(T)=\Z_\ell$.
By Proposition~\ref{prop:CkT-good}, for every odd integer $d$ with $\gcd(d,ab)=1$,
$
d\in G_{(T,\ell)}(a,b)
$ if and only if $
d\in G_{(i,\ell)}(a,b)$  for all $i\in\{0,1,\dots,\ell-1\}$.
In particular, the membership of $d$ in $G_{(T,\ell)}(a,b)$ is independent of the choice of
$
T\in\{1,\dots,\ell-1\}$,
which implies that
$
d\in G_{(0,\ell)}(a,b)$.
Therefore, 
\[
G^O_{(1,\ell)}(a,b)=G^O_{(2,\ell)}(a,b)=\cdots=G^O_{(\ell-1,\ell)}(a,b)\subseteq G^O_{(0,\ell)}(a,b).
\]
\end{proof}

\begin{remark}
The preceding relations are formulated only for the odd parts
$
G^O_{(T,k)}(a,b),
$
and they do not extend automatically to the full sets $G_{(T,k)}(a,b)$. Indeed, for even integers, membership depends not only on the odd prime-power conditions but also on the $2$-adic divisibility of
$
a^{ks+T}+b^{ks+T}$,
which is determined by the parity of the exponent $ks+T$.
\end{remark}

\section{An Explicit Algorithm and Examples} \label{sec3}
Combining the local characterization at odd prime powers, the global $2$-adic criterion for odd integers, and the analysis of the even case from the previous section, we present  a practical procedure for deciding whether a given integer belongs to $G_{(T,k)}(a,b)$ and, when it does, for constructing a corresponding exponent. The procedure is presented in Algorithm~\ref{alg:good}.

\begin{algorithm}[!hbt]  
\caption{Determine whether $d\in G_{(T,k)}(a,b)$ and compute an exponent}
\label{alg:good}

 \small
\begin{algorithmic}[1]
\Require Coprime nonzero integers $a,b$, integers $k\ge 1$, $0\le T<k$, and $d\ge 1$.
\Ensure Decide whether $d\in G_{(T,k)}(a,b)$; if yes, output one integer $s\ge 1$ such that
\[
d\mid \bigl(a^{ks+T}+b^{ks+T}\bigr).\]

\If{$\gcd(d,ab)\ne 1$}
  \State \Return ``No''.
\EndIf

\State Write
$
d=2^\varepsilon d_0$
with $\varepsilon\ge 0$ and $d_0$ odd.

\If{$\varepsilon\ge 1$ and exactly one of $a,b$ is even}
  \State \Return ``No''.
\EndIf

\If{$d_0>1$}
  \State Factor
  $
  d_0=\prod_{i=1}^m p_i^{e_i}$.
  \For{$i=1$ to $m$}
    \State Compute
    $
    t_i=\ord_{p_i^{e_i}}(ab^{-1}),
    \quad
    g_i=\gcd(k,t_i),
    \quad
    L_i=\frac{t_i}{g_i},
    \quad
    c_i=\frac{t_i}{2}-T$,
    where $b^{-1}$ denotes the inverse of $b$ modulo $p_i^{e_i}$.
    \If{$t_i$ is odd or $g_i\nmid c_i$}
      \State \Return ``No''.
    \EndIf
    \State Solve
    $
    ks\equiv c_i \pmod{t_i}$,
    obtaining one congruence class
    $
    s\equiv s_i \pmod{L_i}$.
  \EndFor

  \State Solve the simultaneous congruences
  $
  s\equiv s_i \pmod{L_i},
  \quad 1\le i\le m$.
  \If{the system is inconsistent}
    \State \Return ``No''.
  \EndIf

  \State Set
  $
  L=\lcm(L_1,\dots,L_m)
  =
  \lcm_{1\le i\le m}\left(\frac{t_i}{g_i}\right)$.
  
  \State Let
  $
  s\equiv s_0 \pmod L$
  be the resulting solution class.

\Else
  \State Set $
  s_0=1 $ and $
  L=1$.
\EndIf

\If{$\varepsilon=0$}
  \State Output any positive integer $s\equiv s_0\pmod L$ and \Return ``Yes''.
\EndIf

\If{$a$ and $b$ are odd}
  \If{$\varepsilon=1$}
    \State Output any positive integer $s\equiv s_0\pmod L$ and \Return ``Yes''.
  \EndIf

  \If{$\varepsilon>\nu_2(a+b)$}
    \State \Return ``No''.
  \EndIf

  \If{there exists a positive integer $s\equiv s_0\pmod L$ such that $ks+T$ is odd}
    \State Output such an $s$ and \Return ``Yes''.
  \Else
    \State \Return ``No''.
  \EndIf
\EndIf
\end{algorithmic}
\end{algorithm}

The following example illustrates how  Algorithm  \ref{alg:good} works in practice. The first part shows the odd case, where the membership test is determined entirely by the odd-part congruence condition. The second one shows how the even part is handled by combining the odd-part condition with the additional $2$-adic restriction.

\begin{example}
Let $a=3$, $b=1$, and $k=6$. We illustrate Algorithm~\ref{alg:good} by determining some memberships in $G_{(1,6)}(3,1)$ and $G_{(3,6)}(3,1)$.

\begin{enumerate}[$(1)$]
\item Let $d=7$.  First, consider   $T=1$. Since
$
\ord_7(3)=6$,
we have
\[
t=6,\quad g=\gcd(6,6)=6, \quad \text{and } c=\frac{6}{2}-1=2.
\]
Thus,  we would need to solve
$
6s\equiv 2 \pmod 6$,
which is impossible. Therefore, 
$
7\notin G_{(1,6)}(3,1)$.

Next, consider   $T=3$. We again have
\[
t=6,\quad g=\gcd(6,6)=6,\quad  \text{and } c=\frac{6}{2}-3=0.
\]
Hence, the congruence
$
6s\equiv 0 \pmod 6$ 
is solvable,  which implies that $
7\in G_{(3,6)}(3,1)$.
For example, taking $s=1$, we obtain
$
3^{6\cdot 1+3}+1=3^9+1=19684=7\cdot 2812$,
which is divisible by $7$.

\item Let $d=28=2^2\cdot 7$.
Write $
d=2^\varepsilon d_0$
with $
\varepsilon=2$ and $d_0=7$.
Since $a=3$ and $b=1$ are odd,we have $
\nu_2(a+b)=\nu_2(4)=2$.

For $T=1$, the odd part already fails since
$
7\notin G_{(1,6)}(3,1)$.
Hence, $
28\notin G_{(1,6)}(3,1)$.

For $T=3$, we already know from \textup{(1)} that $
7\in G_{(3,6)}(3,1)$.
Taking $s=1$, we have
$
6s+3=9$,
which is odd. Moreover,
$
\nu_2(3^9+1)=\nu_2(3+1)=2,
$
which implies that 
$
4\mid (3^9+1)$.
Since 
$
7\mid (3^9+1)$,
it follows that
$
28\mid (3^9+1)$. 
Therefore, $
28\in G_{(3,6)}(3,1)$.
\end{enumerate}
Consequently,
$
7,28\in G_{(3,6)}(3,1),
\quad\text{but}\quad
7,28\notin G_{(1,6)}(3,1)$.
\end{example}

The following example, computed using Algorithm~\ref{alg:good}, presents the sets of $(T,6)$-good integers with respect to $a=3$ and $b=1$ in the range $1\le d\le 100$, for $T=0,1,2,3,4,5$.

\begin{example}
Let $a=3$, $b=1$, and $k=6$. For each $T\in\{0,1,2,3,4,5\}$, the set of $(T,6)$-good integers  in the range $1\le d\le 100$ is given in Table~\ref{tab:36-good}.

\begin{table}[!hbt]
\centering
\caption{The sets $G_{(T,6)}(3,1)\cap \{1,2,\dots,100\}$}
\label{tab:36-good}
\begin{tabular}{cl}
\toprule
$T$ & $G_{(T,6)}(3,1)\cap \{1,2,\dots,100\}$ \\
\midrule
$0$ &
$\{1,2,5,10,17,25,29,34,41,50,53,58,73,82,89,97\}$ \\[1mm]

$1$ &
$\{1,2,4,61,67\}$ \\[1mm]

$2$ &
$\{1,2,5,10,17,25,29,34,41,50,53,58,82,89\}$ \\[1mm]

$3$ &
$\{1,2,4,7,14,19,28,31,37,38,43,49,61,62,67,74,76,79,86,98\}$ \\[1mm]

$4$ &
$\{1,2,5,10,17,25,29,34,41,50,53,58,82,89\}$ \\[1mm]

$5$ &
$\{1,2,4,61,67\}$ \\
\bottomrule
\end{tabular}
\end{table}

\end{example}

\section{$q^k$-Cyclotomic Classes on Finite Abelian Groups}\label{sec4}

 In this section, the condition $
d\in G_{(T,k)}(q,1)$ is applied   
to determine when a $q^k$-cyclotomic class  of  finite cyclic and abelian groups is stable under the action $
g\mapsto -q^T g
$. The map
$
g\mapsto -q^T g
$
provides a common setting for the classical actions arising in coding theory. Indeed, when $k=1$ and $T=0$, it reduces to
$
g\mapsto -g$,
which corresponds to the Euclidean case over $\F_q$. When $k=2$ and $T=1$, it becomes
$
g\mapsto -qg$,
which corresponds to the Hermitian case over $\F_{q^2}$. Hence, the notion of type $T$ naturally extends both the Euclidean and Hermitian settings. These results will later be applied to the study of cyclic and abelian codes and their Galois duality in Section \ref{sec5}.

Let $q$ be a prime power and let  $G$ be a finite abelian group written additively such that $\gcd(q,|G|)=1$.  For $g\in G$ and a positive integer $k$, the order of $g$ in $G$ is denoted by $\ord(g)$ and the \emph{$q^k$-cyclotomic class} containing $g$ is defined to be
\[
S_{q^k}(g)=\{q^{kj}g:j\ge 0\}.
\]

The next lemma gives the cardinality of a $q^k$-cyclotomic class in terms of the order of any of its elements; see, for example, \cite{PJ2018}.
 
\begin{lemma}
\label{lem:cosetsize}
Let $q$ be a prime power and let $G$ be a finite abelian group such that $\gcd(q,|G|)=1$. If $g\in G$ has order $d$, then
$
|S_{q^k}(g)|=\ord_d(q^k)$.
\end{lemma}

Let $G$ be a finite abelian group of order $n$ and exponent $N$ such that $\gcd(n,q)=1$. Then also $\gcd(N,q)=1$. For integers $0\le T<k$ and $g\in G$, the $q^k$-cyclotomic class $S_{q^k}(g)$ is said to be \emph{of type $T$} if
\[
S_{q^k}(g)=S_{q^k}(-q^Tg),
\]
and \emph{of type $T'$} otherwise. The following example illustrates the notion of type $T$ for $q^k$-cyclotomic classes.

\begin{example}\label{ex:z35-q3k4}
Let $q=3$, $k=4$, and $G=\Z_{35}$. Then the $q^k$-cyclotomic classes are the $81$-cyclotomic classes in $\Z_{35}$, namely
$
S_{81}(0)=\{0\}$, $
S_{81}(1)=S_{81}(11)=S_{81}(16)=\{1,11,16\}$,
$S_{81}(2)=S_{81}(22)=S_{81}(32)=\{2,22,32\}$, $
S_{81}(3)=S_{81}(33)=S_{81}(13)=\{3,33,13\}$, $
S_{81}(4)=S_{81}(9)=S_{81}(29)=\{4,9,29\}$, $
S_{81}(5)=S_{81}(20)=S_{81}(10)=\{5,20,10\}$, $
S_{81}(6)=S_{81}(31)=S_{81}(26)=\{6,31,26\}$, $
S_{81}(7)=\{7\}$, $
S_{81}(8)=S_{81}(18)=S_{81}(23)=\{8,18,23\}$, $
S_{81}(12)=S_{81}(27)=S_{81}(17)=\{12,27,17\}$, $
S_{81}(14)=\{14\}$, $
S_{81}(15)=S_{81}(25)=S_{81}(30)=\{15,25,30\}$, $
S_{81}(19)=S_{81}(34)=S_{81}(24)=\{19,34,24\}$, $
S_{81}(21)=\{21\}$, $
S_{81}(28)=\{28\}$.

For $T=0$, the action is $g\mapsto -g$. In $\Z_{35}$,  we have  $
-0=0 \in S_{81}(0)$, 
$ -1=34 \notin S_{81}(1)$,  
$ -2=33 \notin S_{81}(2) $,  
$ -3=32 \notin S_{81}(3)$,  
$ -4=31 \notin S_{81}(4)$, 
$ -5=30 \notin S_{81}(5)$, 
 $
-6=29 \notin S_{81}(6)$, 
$ -7=28 \notin S_{81}(7)$,  
$ -8=27 \notin S_{81}(8)$,  
$ -12=23 \notin S_{81}(12)$,  
$ -14=21 \notin S_{81}(14)$,  
$ -15=20 \notin S_{81}(15)$, 
$-19=16 \notin S_{81}(19)$,  
$ -21=14 \notin S_{81}(21)$,  
$ -28=7\notin S_{81}(28)$. 
Hence, only $S_{81}(0)$ is of type $0$, and every other $81$-cyclotomic class is of type $0'$.

For $T=1$, the action is $g\mapsto -3g$. In $\Z_{35}$, we have 
  $
-3\cdot 0=0 \in S_{81}(0)$,  
$ -3\cdot 1=32 \notin S_{81}(1) $,  
$ -3\cdot 2=29 \notin S_{81}(2)$,  
$ -3\cdot 3=26  \notin S_{81}(3)$,  
$ -3\cdot 4=23 \notin S_{81}(4)$,
 $ 
-3\cdot 5=20 \in S_{81}(5)$,  
$ -3\cdot 6=17 \notin S_{81}(6)$,  
$ -3\cdot 7=14  \notin S_{81}(7)$,  
$ -3\cdot 8=11  \notin S_{81}(8)$,  
$ -3\cdot 12=34 \notin S_{81}(12)$, 
 $
-3\cdot 14=28  \notin S_{81}(14)$,  
$ -3\cdot 15=25 \in S_{81}(15)$,  
$ -3\cdot 19=13 \notin S_{81}(19)$,  
$ -3\cdot 21=7 \notin S_{81}(21)$,  
$ -3\cdot 28=21 \notin S_{81}(28)
 $.
Hence, 
  $S_{81}(0)$, $S_{81}(5)$, and $S_{81}(15)$ are of type $1$ and all the remaining $81$-cyclotomic classes are of type $1'$.

For $T=2$ and $T=3$, the types of the $81$-cyclotomic classes are determined in the same manner. The resulting type distribution of $81$-cyclotomic classes in $\Z_{35}$ is presented in Table~\ref{T-35}.

\begin{table}[!hbt]
\centering
\caption{Type distribution of $81$-cyclotomic classes in $\Z_{35}$}  \label{T-35} 
\begin{tabular}{l|cccc}
\hline
\text{$81$-cyclotomic class} & $T=0$ & $T=1$ & $T=2$ & $T=3$\\
\hline
$S_{81}(0)=\{0\}$              & $0$  & $1$  & $2$  & $3$  \\
$S_{81}(1)=\{1,11,16\}$        & $0'$ & $1'$ & $2'$ & $3'$ \\
$S_{81}(2)=\{2,22,32\}$        & $0'$ & $1'$ & $2'$ & $3'$ \\
$S_{81}(3)=\{3,33,13\}$        & $0'$ & $1'$ & $2'$ & $3'$ \\
$S_{81}(4)=\{4,9,29\}$         & $0'$ & $1'$ & $2'$ & $3'$ \\
$S_{81}(5)=\{5,20,10\}$        & $0'$ & $1$  & $2'$ & $3$  \\
$S_{81}(6)=\{6,31,26\}$        & $0'$ & $1'$ & $2'$ & $3'$ \\
$S_{81}(7)=\{7\}$              & $0'$ & $1'$ & $2$  & $3'$ \\
$S_{81}(8)=\{8,18,23\}$        & $0'$ & $1'$ & $2'$ & $3'$ \\
$S_{81}(12)=\{12,27,17\}$      & $0'$ & $1'$ & $2'$ & $3'$ \\
$S_{81}(14)=\{14\}$            & $0'$ & $1'$ & $2$  & $3'$ \\
$S_{81}(15)=\{15,25,30\}$      & $0'$ & $1$  & $2'$ & $3$  \\
$S_{81}(19)=\{19,34,24\}$      & $0'$ & $1'$ & $2'$ & $3'$ \\
$S_{81}(21)=\{21\}$            & $0'$ & $1'$ & $2$  & $3'$ \\
$S_{81}(28)=\{28\}$            & $0'$ & $1'$ & $2$  & $3'$ \\
\hline
\end{tabular}
\end{table}
\end{example}

\subsection{Characterization}
The characterization of $q^k$-cyclotomic classes of type $T$ is given in this subsection. It shows that the type of a class is determined entirely by the order of its elements through membership in the  set $G_{(T,k)}(q,1)$.

\begin{theorem}\label{thm:abelian-type}
Let $q$ be a prime power and let   $k$ and $T$ be integers such that $0\le T<k$.
Let $G$ be a finite abelian group such that $\gcd(q,|G|)=1$ and  let $g\in G$ be such that $\ord(g)=d$. Then $
S_{q^k}(g)$    is of type $T$
if and only if $
d\in G_{(T,k)}(q,1)$.
In particular, the type of $S_{q^k}(g)$ depends only on the order of $g$.
\end{theorem}

\begin{proof}
Assume first that $S_{q^k}(g)$ is of type $T$. Then
$
S_{q^k}(g)=S_{q^k}(-q^Tg)$.
Since $g\in S_{q^k}(g)$, it follows that
$
g\in S_{q^k}(-q^Tg).
$
Hence,  there exists an integer $j\ge 0$ such that
$
q^{kj}(-q^Tg)=g$.
Equivalently, $
(q^{kj+T}+1)g=0$.
Since $g$ has order $d$, this holds if and only if $
d\mid (q^{kj+T}+1)$.
If $j\ge 1$, then $
d\in G_{(T,k)}(q,1)$. Assume that $j=0$. Then $
d\mid (q^T+1)$.
Let $m=\ord_d(q^k)$. Then $m\ge 1$ and
$
q^{km+T}\equiv q^T\equiv -1\pmod d$.
Hence, 
$
d\mid (q^{km+T}+1)$,
which implies that $
d\in G_{(T,k)}(q,1)$.

Conversely, assume that $
d\in G_{(T,k)}(q,1)$.
Then there exists an integer $s\ge 1$ such  that $
d\mid (q^{ks+T}+1)$.
Equivalently,
$
(q^{ks+T}+1)g=0$.
Thus, $
q^{ks}(-q^Tg)=g$,
which implies  that
$
g\in S_{q^k}(-q^Tg)$.
Hence,  $
S_{q^k}(g)=S_{q^k}(-q^Tg)$.
Therefore, $S_{q^k}(g)$ is of type $T$.

The final statement is immediate since the right hand side depends only on $d=\ord(g)$.
\end{proof}

\begin{example}
Let $q=3$, $k=4$, and $G=\Z_{35}$. By Example~\ref{ex:z35-q3k4}, the $q^k$-cyclotomic classes, namely the $81$-cyclotomic classes in $\Z_{35}$, are known explicitly. We  determine their types using Theorem~\ref{thm:abelian-type} and the arithmetic sets $G_{(T,4)}(3,1)$. We note that  $
\ord(0)=1$, $
\ord(1)=\ord(2)=\ord(3)=\ord(4)=\ord(6)=\ord(8)=\ord(12)=\ord(19)=35$, $
\ord(5)=\ord(15)=7$,
and
$
\ord(7)=\ord(14)=\ord(21)=\ord(28)=5$.
Since 
$
1\in G_{(T,4)}(3,1)$ for all $T=0,1,2,3$, 
$35\notin G_{(0,4)}(3,1)$,
$35\notin G_{(1,4)}(3,1)$,
$35\notin G_{(2,4)}(3,1)$,
$35\notin G_{(3,4)}(3,1)$,
$7\notin G_{(0,4)}(3,1)$,
$7\in G_{(1,4)}(3,1)$, 
$7\notin G_{(2,4)}(3,1)$,
$7\in G_{(3,4)}(3,1)$,
and
$
5\notin G_{(0,4)}(3,1)$,
$5\notin G_{(1,4)}(3,1)$,
$5\in G_{(2,4)}(3,1)$,
$5\notin G_{(3,4)}(3,1)$,  by Theorem~\ref{thm:abelian-type},  it follows that  
$
S_{81}(0)$
is of types $0,1,2,3$, 
$
S_{81}(5)$ and  $ S_{81}(15)$
are of types $1$ and $3$, and of types $0'$ and $2'$, 
$
S_{81}(7), S_{81}(14), S_{81}(21)$, and $S_{81}(28)
$
are of type $2$, and of types $0',1',3'$,  and
$
S_{81}(1), S_{81}(2), S_{81}(3), S_{81}(4), S_{81}(6), S_{81}(8), S_{81}(12)$, and $S_{81}(19)
$
are of types $0',1',2',3'$.
\end{example}

 Now, we  examine how the types of cyclotomic classes interact as $T$ varies. To describe this propagation explicitly, we associate to each $T$ a subset $C_k(T)$ of $\Z_k$, which records all types forced by type $T$.

\begin{theorem}\label{thm:CkT-propagation}
Let $q$ be a prime power and let  $G$ be a finite abelian group  such that $\gcd(q, |G|)=1$. Let  $g\in G$ and let $k$ and $T$ be integers such that $0\le T<k$. Then $
S_{q^k}(g)\text{ is of type }T
$
if and only if
$
S_{q^k}(g) $ is of type $\ell$ for every $\ell\in C_k(T)$.
\end{theorem}

\begin{proof}
Assume that $S_{q^k}(g)$ is of type $T$. Then
$
S_{q^k}(g)=S_{q^k}(-q^Tg)$ and there exists an integer $s\ge 0$ such that
$
q^{ks}g=-q^Tg$.
Multiplying both sides by $-q^T$, we obtain
$
q^{ks}(-q^Tg)=q^{2T}g$,
and hence, $
S_{q^k}(g)=S_{q^k}(q^{2T}g)$.
Multiplying once more by $-q^T$, we get
$
q^{ks}(q^{2T}g)=-q^{3T}g$,
which implies that 
\[
S_{q^k}(g)=S_{q^k}(-q^{3T}g).
\]
Iterating this argument, we obtain
\[
S_{q^k}(g)=S_{q^k}(-q^{(2i+1)T}g)
 \] for all $i\ge 0$.
Thus,  $S_{q^k}(g)$ is of type $\ell$ for every $\ell\in C_k(T)$.

Conversely, since $T\in C_k(T)$, the reverse implication is immediate.
\end{proof}

Combining Theorem~\ref{thm:CkT-propagation} and Lemma~\ref{lem:CkT-symmetry}, we obtain the following consequence.

\begin{corollary}\label{cor:type-k-minus-T}
Let $q$ be a prime power and let  $G$ be a finite abelian group  such that $\gcd(q, |G|)=1$. Let  $g\in G$ and let $k$ and $T$ be integers such that $0\le T<k$.  The $
S_{q^k}(g)$  is of type $T$  if and only if $
S_{q^k}(g)$  is of type $k-T$.
\end{corollary}

\subsection{Enumeration}
The characterization obtained above leads directly to enumeration formulas for $q^k$-cyclotomic classes of type $T$ and type $T'$. Since the type of a class is determined by the order of its elements, it is enough to count how many classes arise from each divisor of the exponent of $G$.

Let $G$ be a finite abelian group of exponent $N$. For each divisor $d$ of $N$, let
$
\mathcal{N}_G(d)=\bigl|\{g\in G:\ord(g)=d\}\bigr|$ and let  $\varphi(d)$ denote the \emph{Euler phi function}.
The quantity $\mathcal{N}_G(d)$ is well known; see, for example, \cite{B1997}.  Especially,  we have   $\mathcal{N}_G(d)=\varphi(d)$  if $G$ is cyclic.

\begin{proposition}\label{prop:type-count}
Let $q$ be a prime power and let $G$ be a finite abelian group of exponent $N$ such that $\gcd(q,N)=1$.  Let $k$ and $T$ be integers such that $0\le T<k$. Then the number of $q^k$-cyclotomic classes of type $T$ is
\begin{equation}\label{eq:typeT-count}
\mathfrak{N}_T(G,q^k)
=
\sum_{\substack{d\mid N\\ d\in G_{(T,k)}(q,1)}}
\frac{\mathcal{N}_G(d)}{\ord_d(q^k)},
\end{equation}
and the number of $q^k$-cyclotomic classes of type $T'$ is
\begin{equation}\label{eq:typeTp-count}
\mathfrak{N}_{T'}(G,q^k)
=
\sum_{\substack{d\mid N\\ d\notin G_{(T,k)}(q,1)}}
\frac{\mathcal{N}_G(d)}{\ord_d(q^k)}.
\end{equation}
\end{proposition}

\begin{proof}
Let $d$ be a divisor of $N$. By Theorem~\ref{thm:abelian-type}, a $q^k$-cyclotomic class containing an element of order $d$ is of type $T$ if and only if
$
d\in G_{(T,k)}(q,1).
$
Moreover, by Lemma~\ref{lem:cosetsize}, every such class has size $\ord_d(q^k)$. Since there are $\mathcal{N}_G(d)$ elements of order $d$, the number of $q^k$-cyclotomic classes whose elements have order $d$ is
\[
\frac{\mathcal{N}_G(d)}{\ord_d(q^k)}.
\]
Summing over all divisors $d$ of $N$ such that $d\in G_{(T,k)}(q,1)$ gives \eqref{eq:typeT-count}, while summing over all divisors $d$ of $N$ such that $d\notin G_{(T,k)}(q,1)$ gives \eqref{eq:typeTp-count}.
\end{proof}

We illustrate the preceding counting formulas by computing the numbers of $q^k$-cyclotomic classes of types $T$ and $T'$   in finite cyclic groups.

 \begin{example}\label{ex:enumeration-35}
We illustrate Proposition~\ref{prop:type-count} for the cyclic groups $
\Z_{35}$, $\Z_{70}$,   $\Z_{140}$,    and  $\Z_{280}$ 
with $q=3$ and $k=4$. For a cyclic group $\Z_n$, the number of elements of additive order $d\mid n$ is $
\mathcal{N}_{\Z_n}(d)=\varphi(d)$.
Hence, by Proposition~\ref{prop:type-count}, we have 
\[
\mathfrak N_T(\Z_n,3^4)
=
\sum_{\substack{d\mid n\\ d\in G_{(T,4)}(3,1)}}
\frac{\varphi(d)}{\ord_d(3^4)}.
\]
Now, we apply this formula for each $n\in\{35,70\}$ and each $T\in\{0,1,2,3\}$.

\medskip
\noindent
\textbf{Case $n=35$.}
Since the relevant divisors are
$
G_{(0,4)}(3,1)\cap\{d:d\mid 35\}=\{1\}$, $
G_{(1,4)}(3,1)\cap\{d:d\mid 35\}=\{1,7\}$, $
G_{(2,4)}(3,1)\cap\{d:d\mid 35\}=\{1,5\}$, $
G_{(3,4)}(3,1)\cap\{d:d\mid 35\}=\{1,7\}$,
we obtain
\[
\mathfrak N_0(\Z_{35},3^4)
=
\sum_{\substack{d\mid 35\\ d\in G_{(0,4)}(3,1)}}
\frac{\varphi(d)}{\ord_d(3^4)}
=
\frac{\varphi(1)}{\ord_1(3^4)}
=
\frac{1}{1}
=
1,
\]
\[
\mathfrak N_1(\Z_{35},3^4)
=
\sum_{\substack{d\mid 35\\ d\in G_{(1,4)}(3,1)}}
\frac{\varphi(d)}{\ord_d(3^4)}
=
\frac{\varphi(1)}{\ord_1(3^4)}
+\frac{\varphi(7)}{\ord_7(3^4)}
=
\frac{1}{1}+\frac{6}{3}
=
3,
\]
\[
\mathfrak N_2(\Z_{35},3^4)
=
\sum_{\substack{d\mid 35\\ d\in G_{(2,4)}(3,1)}}
\frac{\varphi(d)}{\ord_d(3^4)}
=
\frac{\varphi(1)}{\ord_1(3^4)}
+\frac{\varphi(5)}{\ord_5(3^4)}
=
\frac{1}{1}+\frac{4}{1}
=
5,
\] and 
\[
\mathfrak N_3(\Z_{35},3^4)
=
\sum_{\substack{d\mid 35\\ d\in G_{(3,4)}(3,1)}}
\frac{\varphi(d)}{\ord_d(3^4)}
=
\frac{\varphi(1)}{\ord_1(3^4)}
+\frac{\varphi(7)}{\ord_7(3^4)}
=
\frac{1}{1}+\frac{6}{3}
=
3.
\]

\medskip
\noindent
\textbf{Case $n=70$.}
Since
$
G_{(0,4)}(3,1)\cap\{d:d\mid 70\}=\{1,2\}$,
$
G_{(1,4)}(3,1)\cap\{d:d\mid 70\}=\{1,2,7,14\}$,
$
G_{(2,4)}(3,1)\cap\{d:d\mid 70\}=\{1,2,5,10\}$, $
G_{(3,4)}(3,1)\cap\{d:d\mid 70\}=\{1,2,7,14\}$,
it follows that 
\[
\mathfrak N_0(\Z_{70},3^4)
=
\frac{\varphi(1)}{\ord_1(3^4)}
+\frac{\varphi(2)}{\ord_2(3^4)}
=
\frac{1}{1}+\frac{1}{1}
=
2,
\]
\[
\mathfrak N_1(\Z_{70},3^4)
=
\frac{\varphi(1)}{\ord_1(3^4)}
+\frac{\varphi(2)}{\ord_2(3^4)}
+\frac{\varphi(7)}{\ord_7(3^4)}
+\frac{\varphi(14)}{\ord_{14}(3^4)}
=
\frac{1}{1}+\frac{1}{1}+\frac{6}{3}+\frac{6}{3}
=
6,
\]
\[
\mathfrak N_2(\Z_{70},3^4)
=
\frac{\varphi(1)}{\ord_1(3^4)}
+\frac{\varphi(2)}{\ord_2(3^4)}
+\frac{\varphi(5)}{\ord_5(3^4)}
+\frac{\varphi(10)}{\ord_{10}(3^4)}
=
\frac{1}{1}+\frac{1}{1}+\frac{4}{1}+\frac{4}{1}
=
10,
\]
and 
\[
\mathfrak N_3(\Z_{70},3^4)
=
\frac{\varphi(1)}{\ord_1(3^4)}
+\frac{\varphi(2)}{\ord_2(3^4)}
+\frac{\varphi(7)}{\ord_7(3^4)}
+\frac{\varphi(14)}{\ord_{14}(3^4)}
=
\frac{1}{1}+\frac{1}{1}+\frac{6}{3}+\frac{6}{3}
=
6.
\]
 
 The cases $n=140$ and $n=280$ can be computed in the same manner.
Therefore, the numbers of $81$-cyclotomic classes of type $T$ in these cyclic groups are summarized as follows:
\[
\begin{array}{c|cccc}
G & \mathfrak N_0(G,3^4) & \mathfrak N_1(G,3^4) & \mathfrak N_2(G,3^4) & \mathfrak N_3(G,3^4)\\
\hline
\Z_{35}  & 1 & 3  & 5  & 3 \\
\Z_{70}  & 2 & 6  & 10 & 6 \\
\Z_{140} & 2 & 12 & 10 & 12 \\
\Z_{280} & 2 & 12 & 10 & 12
\end{array}
\]

Similarly, the numbers of $81$-cyclotomic classes of type $T'$ are given by
\[
\begin{array}{c|cccc}
G & \mathfrak N_{0'}(G,3^4) & \mathfrak N_{1'}(G,3^4) & \mathfrak N_{2'}(G,3^4) & \mathfrak N_{3'}(G,3^4)\\
\hline
\Z_{35}  & 14 & 12 & 10 & 12 \\
\Z_{70}  & 28 & 24 & 20 & 24 \\
\Z_{140} & 38 & 28 & 30 & 28 \\
\Z_{280} & 78 & 68 & 70 & 68
\end{array}
\]
\end{example}

\section{Applications: Galois Duality of Cyclic Codes}
\label{sec5}

In this section, the arithmetic and orbit theoretic results developed in Section~\ref{sec4} are applied to cyclic codes over $\F_{q^k}$. The main point is that the action
$
i\mapsto -q^Ti$
on $q^k$-cyclotomic classes modulo $n$ describes the $\theta$-reciprocal behavior of the irreducible factors of $x^n-1$. We first study Galois linear complementary dual  (LCD) cyclic codes in the general setting, including repeated-root codes. We then turn to Galois self-dual repeated-root cyclic codes, where the correct description is orbit-wise in general and specializes to the familiar pairwise form in the involutory case. These families of codes are significant because they extend the classical Euclidean and Hermitian settings and play an important role in the study of algebraic duality and symmetry in coding theory.  The applications developed here are    general theory, and it would be of further interest to investigate specific code lengths and classes of codes for which the arithmetic properties of $G_{(T,k)}(q,1)$ give more explicit structural descriptions and sharper enumeration formulas.

\subsection{Cyclic codes over finite fields}

In this subsection, we recall the basic definitions and notation for cyclic codes over finite fields required in the study of Galois duality of cyclic codes. In particular, we review their description in terms of ideals in $\F_{q^k}[x]/\langle x^n-1\rangle$, together with the notions of generator polynomial, check polynomial, and Galois duality.

A \emph{linear code}  $C$ of length $n$ over $\F_{q^k}$ is a vector subspace
of $ \F_{q^k}^n$.
A linear code $C$ is called \emph{cyclic} if
\[ (c_{n-1},c_0,c_1,\dots,c_{n-2})\in C \text{ for all }
(c_0,c_1,\dots,c_{n-1})\in C.
\]
 Via the linear isomorphism
$
\pi:\F_{q^k}^n \rightarrow \frac{\F_{q^k}[x]}{\langle x^n-1\rangle}
$
defined by
\[
(c_0,c_1,\dots,c_{n-1}) \mapsto c_0+c_1x+\cdots+c_{n-1}x^{n-1},
\]
cyclic codes of length $n$ over $\F_{q^k}$  are precisely the ideals of the quotient ring
$
\frac{\F_{q^k}[x]}{\langle x^n-1\rangle}$.
Hence, every cyclic code has the form $
C\cong\langle g(x)\rangle$,
where $g(x)$ is the unique monic divisor of $x^n-1$ of minimal degree in $\pi(C)$. This polynomial is called the \emph{generator polynomial} of $C$; see, for example, \cite{JLX2011}.

For integers $0\le T<k$, let
\[
\theta(a)=a^{q^T} 
\]
be the corresponding {\em Galois automorphism} of $\F_{q^k}$. The associated \emph{Galois form}, or specifically the \emph{$\theta$-form}, on $\F_{q^k}^n$ is defined by
\[
[u,v]_{\theta}=\sum_{i=0}^{n-1}u_i\,\theta(v_i),
\]
for all 
$u=(u_0,\dots,u_{n-1})$ and $ v=(v_0,\dots,v_{n-1})$ in $ \F_{q^k}^n$.
The corresponding dual code is called the \emph{Galois dual}, or specifically the \emph{$\theta$-dual} of $C$ and it is given  by
\[
C^{\perp_\theta} =\{ v\in \F_{q^k}^n: [v,c]_{\theta} =0 \text{ for all } c\in C\}.
\]
For a cyclic code $C=\langle g(x)\rangle$ of length $n$ over $\F_{q^k}$ and 
$
h(x)=\frac{x^n-1}{g(x)}
$, we have  that  $C^{\perp_\theta}$  is again  a cyclic code  of length $n$ over $\F_{q^k}$ and 
\[ 
C^{\perp_\theta}=\langle h^{\ast,\theta}(x)\rangle.
\]

A cyclic code $C$ is called \emph{Galois LCD}, or specifically the \emph{$\theta$-LCD}  if
\[
C\cap C^{\perp_\theta}=\{0\},
\]
and it is called \emph{Galois self-dual}, or specifically the \emph{$\theta$-self-dual} if
\[
C=C^{\perp_\theta}.
\]
It follows that $C$ is  $\theta$-LCD  if and only if 
\begin{align}
  \gcd(g(x),  h^{\ast,\theta}(x) )=1\label{eq-LCD},
\end{align}
and $C$ is  $\theta$-self-dual if and only if 
\begin{align}
  g(x)=  h^{\ast,\theta}(x) \label{eq-SD},
\end{align}

\subsection{\texorpdfstring{$\theta$}{theta}-self-reciprocal irreducible factors of $x^n-1$ over $\F_{q^k}$}

We consider the simple root case $\gcd(n,q)=1$. For integers $0\leq T<k$, let $\theta$ be the Galois automorphism of $\F_{q^k}$ defined by
\[
\theta(\alpha)=\alpha^{q^T}.
\]
For a polynomial
$
f(x)=f_0+f_1x+\cdots+f_mx^m\in \F_{q^k}[x]
$
with $f_0, f_m\ne 0$, define its \emph{$\theta$-reciprocal} by
\[
f^{\ast,\theta}(x)=\theta(f_0)^{-1}x^m\theta\bigl(f(x^{-1})\bigr).
\]
A monic polynomial $f$ with nonzero constant term is called \emph{$\theta$-self-reciprocal} if
$
f^{\ast,\theta}(x)=f(x)$.

Fix $n\ge 1$ with $\gcd(n,q)=1$, let $\zeta$ be a primitive $n$th root of unity in an extension field of   $\F_{q^k}$. For each $q^k$-cyclotomic class $Q\subseteq \Z_n$, define
\[
f_Q(x)=\prod_{i\in Q}(x-\zeta^i).
\]
Then $f_Q(x)$ is a monic irreducible factor of $x^n-1$ over $\F_{q^k}$, and every monic irreducible factor of $x^n-1$  is obtained  in this way.

\begin{theorem}\label{thm:reciprocal-coset}
For every $q^k$-cyclotomic class $Q\subseteq \Z_n$,
\[
f_Q^{\ast,\theta}(x)=f_{-q^TQ}(x),
\]
where
$
-q^TQ=\{-q^Ti \bmod n : i\in Q\}$.
Consequently, $f_Q$ is $\theta$-self-reciprocal if and only if
\[
Q=-q^TQ,
\]
or equivalently, if and only if every element of $Q$ has additive order in $G_{(T,k)}(q,1)$.
\end{theorem}

\begin{proof}
The roots of $f_Q(x)$ are precisely the elements $\zeta^i$ with $i\in Q$. Passing to $x^{-1}$ replaces each root $\zeta^i$ by $\zeta^{-i}$, and applying $\theta$ to the coefficients raises the roots to the $q^T$th power. Hence, the roots of $f_Q^{\ast,\theta}(x)$ are exactly
$
(\zeta^{-i})^{q^T}=\zeta^{-q^Ti}$ for all $i\in Q$,
which is precisely the root set attached to $-q^TQ$. Therefore,
\[
f_Q^{\ast,\theta}(x)=f_{-q^TQ}(x).
\]
The final statement follows from Theorem~\ref{thm:abelian-type} with $G=\Z_n$.
\end{proof}

To describe the full factorization of $x^n-1$, let $\sigma$ be the permutation of the set of $q^k$-cyclotomic classes modulo $n$ defined by
\[
\sigma(Q)=-q^TQ.
\]

\begin{corollary}\label{cor:factorization-orbits}
Let $\Omega$ be a set of representatives of the $\sigma$-orbits on the set of $q^k$-cyclotomic classes modulo $n$. For each $Q\in \Omega$, let $
\mathcal O(Q)=\{Q,\sigma(Q),\sigma^2(Q),\dots,\sigma^{r_Q-1}(Q)\}$
be the $\sigma$-orbit of $Q$, where $r_Q=|\mathcal O(Q)|$. Then
\[
x^n-1=\prod_{Q\in \Omega}\ \prod_{R\in \mathcal O(Q)} f_R(x).
\]
Moreover, $f_Q(x)$ is $\theta$-self-reciprocal if and only if $r_Q=1$.
\end{corollary}

\begin{proof}
By Theorem~\ref{thm:reciprocal-coset},
$
f_Q^{\ast,\theta}(x)=f_{\sigma(Q)}(x)$.
Hence,  repeated application of $\ast,\theta$ permutes the irreducible factors according to the $\sigma$-orbits on the set of $q^k$-cyclotomic classes modulo $n$. Partitioning the set of cyclotomic classes into these orbits allows the stated factorization. The final assertion is immediate from Theorem~\ref{thm:reciprocal-coset}.
\end{proof}

\begin{corollary}\label{cor:factorization-separate}
Let $\Omega_1$ be the set of representatives $Q$ of the $\sigma$-orbits on the set of $q^k$-cyclotomic classes modulo $n$ such that $|\mathcal O(Q)|=1$, and let $\Omega_2$ be the set of representatives $Q$ such that $|\mathcal O(Q)|>1$. Then
\[
x^n-1=
\left(\prod_{Q\in \Omega_1} f_Q(x)\right)
\left(\prod_{Q\in \Omega_2}\ \prod_{R\in \mathcal O(Q)} f_R(x)\right).
\]
Moreover, the factors $f_Q(x)$ with $Q\in \Omega_1$ are precisely the monic irreducible $\theta$-self-reciprocal factors of $x^n-1$, whereas each factor belonging to an orbit represented by some $Q\in \Omega_2$ is not $\theta$-self-reciprocal.
\end{corollary}

\begin{proof}
By Corollary~\ref{cor:factorization-orbits},
\[
x^n-1=\prod_{Q\in \Omega}\ \prod_{R\in \mathcal O(Q)} f_R(x),
\]
where $\Omega$ is a set of representatives of the $\sigma$-orbits. Partition $\Omega$ as
\[
\Omega=\Omega_1\cup \Omega_2,
\]
where $\Omega_1$ consists of the representatives of the orbits of size $1$, and $\Omega_2$ consists of the representatives of the orbits of size greater than $1$. This gives the stated factorization.

If $Q\in \Omega_1$, then $|\mathcal O(Q)|=1$, so $\sigma(Q)=Q$. Hence, by Theorem~\ref{thm:reciprocal-coset},
\[
f_Q^{\ast,\theta}(x)=f_{\sigma(Q)}(x)=f_Q(x),
\]
and  $f_Q(x)$ is therefore $\theta$-self-reciprocal.

Conversely, if $Q\in \Omega_2$, then $|\mathcal O(Q)|>1$, so $\sigma(Q)\neq Q$. Again by Theorem~\ref{thm:reciprocal-coset},
\[
f_Q^{\ast,\theta}(x)=f_{\sigma(Q)}(x)\neq f_Q(x),
\]
and hence, $f_Q(x)$ is not $\theta$-self-reciprocal. The same argument applies to every factor $f_R(x)$ with $R\in \mathcal O(Q)$.
\end{proof}

\begin{remark}
By Theorem~\ref{thm:reciprocal-coset}, a $q^k$-cyclotomic class $Q$ belongs to $\Omega_1$ if and only if every element of $Q$ has additive order in $G_{(T,k)}(q,1)$. Hence, the self-reciprocal part of the factorization of $x^n-1$ is determined precisely by the arithmetic set $G_{(T,k)}(q,1)$.
\end{remark}

The following example illustrates how the factorization of $x^{17}-1$ over $\F_{2^4}$ is organized by the $\sigma$-orbits corresponding to the values $T=0,1,2,$ and $3$.
\begin{example}\label{ex:x17minus1-F16}
Let $\zeta$ be a primitive $17$th root of unity  in some extension field of  $\F_{2^4}$. The $2^4$-cyclotomic classes modulo $17$ are
$
Q_0=\{0\}$, $
Q_1=\{1,16\}$,
$
Q_2=\{2,15\}$,
$
Q_3=\{3,14\}$,
$
Q_4=\{4,13\}$,
$
Q_5=\{5,12\}$,
$
Q_6=\{6,11\}$,
$
Q_7=\{7,10\}$,
$
Q_8=\{8,9\}$.
For each $i\in\{1,2,\dots,8\}$, let
\[
f_i(x)=f_{Q_i}(x)=\prod_{j\in Q_i}(x-\zeta^j)
=(x-\zeta^i)(x-\zeta^{-i}).
\]
Then each $f_i(x)$ is an irreducible quadratic polynomial over $\F_{2^4}$, and
\[
x^{17}-1=(x-1)\prod_{i=1}^8 f_i(x).
\]

We now group this factorization according to the action
$
Q\longmapsto -2^TQ$.

\noindent
\textbf{Case $T=0$.}
Then $
\sigma_0(Q_i)=-Q_i=Q_i$ for all $i\in \{0,1,\dots,8\}$.  Hence,  every  $\sigma_0$-orbits have size $1$, namely
\[
\{Q_0\},\{Q_1\},\{Q_2\},\{Q_3\},\{Q_4\},\{Q_5\},\{Q_6\},\{Q_7\},\{Q_8\}.
\]
Therefore,
\[
x^{17}-1=(x-1)f_1(x)f_2(x)f_3(x)f_4(x)f_5(x)f_6(x)f_7(x)f_8(x),
\]
and every irreducible factor is $\theta$-self-reciprocal.

\noindent
\textbf{Case $T=1$.} Then 
$
\sigma_1(Q)=-2Q=2Q$. Thus
$
Q_1\mapsto Q_2\mapsto Q_4\mapsto Q_8\mapsto Q_1$,
and
$
Q_3\mapsto Q_6\mapsto Q_5\mapsto Q_7\mapsto Q_3$.
Hence, the $\sigma_1$-orbits are
\[
\{Q_0\},\quad \{Q_1,Q_2,Q_4,Q_8\},\quad \{Q_3,Q_5,Q_6,Q_7\}.
\]
Therefore,
\[
x^{17}-1
=
(x-1)\,
\bigl(f_1(x)f_2(x)f_4(x)f_8(x)\bigr)\,
\bigl(f_3(x)f_5(x)f_6(x)f_7(x)\bigr).
\]
In particular, the only $\theta$-self-reciprocal irreducible factor is $x-1$.

\noindent
\textbf{Case $T=2$.}
Then $
\sigma_2(Q)=-2^2Q=-4Q=4Q$.
Thus
$
Q_1\leftrightarrow Q_4, 
Q_2\leftrightarrow Q_8, 
Q_3\leftrightarrow Q_5, 
Q_6\leftrightarrow Q_7$.
Hence,  the $\sigma_2$-orbits are
\[
\{Q_0\},\{Q_1,Q_4\}, \{Q_2,Q_8\},
\{Q_3,Q_5\}, \{Q_6,Q_7\}.
\]
Therefore,
\[
x^{17}-1
=
(x-1)\,
\left(f_1(x)f_4(x)\right)\,\left(
f_2(x)f_8(x)\right)\, \left(
f_3(x)f_5(x)\right)\,
\left(f_6(x)f_7(x)\right).
\]
Thus,   the only $\theta$-self-reciprocal irreducible factor is again $x-1$.

\noindent
\textbf{Case $T=3$.}
Then 
$
\sigma_3(Q)=-2^3Q=-8Q=8Q$.
Thus,
$
Q_1\mapsto Q_8\mapsto Q_4\mapsto Q_2\mapsto Q_1$ 
and
$
Q_3\mapsto Q_7\mapsto Q_5\mapsto Q_6\mapsto Q_3$.
Hence, the $\sigma_3$-orbits are
\[
\{Q_0\}, \{Q_1,Q_2,Q_4,Q_8\}, \{Q_3,Q_5,Q_6,Q_7\}.
\]
Therefore,
\[
x^{17}-1
=
(x-1)\,
\bigl(f_1(x)f_2(x)f_4(x)f_8(x)\bigr)\,
\bigl(f_3(x)f_5(x)f_6(x)f_7(x)\bigr).
\]
Thus, the only $\theta$-self-reciprocal irreducible factor is again $x-1$.
\end{example}

\subsection{Galois LCD cyclic codes}

  Let
$
q=p^m$
be a prime power  and let $
n=n_0p^r$
with $r\ge 0$ and $\gcd(n_0,p)=1$. Then
\[
x^n-1=\bigl(x^{n_0}-1\bigr)^{p^r}.
\]
Let $\sigma$ be the permutation of the set of $q^k$-cyclotomic classes modulo $n_0$ defined by
$
\sigma(Q)=-q^TQ$.
Using Corollary~\ref{cor:factorization-separate}, write
\[
x^{n_0}-1=
\left(\prod_{Q\in \Omega_1} f_Q(x)\right)
\left(\prod_{Q\in \Omega_2}\ \prod_{R\in \mathcal O(Q)} f_R(x)\right),
\]
where $\Omega_1$ is the set of representatives of the $\sigma$-orbits of size $1$, and $\Omega_2$ is the set of representatives of the $\sigma$-orbits of size greater than $1$. Hence,
\begin{align} \label{xn-1-expan}
x^n-1=
\left(\prod_{Q\in \Omega_1} f_Q(x)^{p^r}\right)
\left(\prod_{Q\in \Omega_2}\ \prod_{R\in \mathcal O(Q)} f_R(x)^{p^r}\right).
\end{align}
For each $Q\in \Omega_2$, define
\[
F_Q(x):=\prod_{R\in\mathcal O(Q)} f_R(x).
\]
Then
$
F_Q^{\ast,\theta}(x)=F_Q(x)
$
and 
\begin{align}\label{eq-xn-1}
x^n-1=
\left(\prod_{Q\in \Omega_1} f_Q(x)^{p^r}\right)
\left(\prod_{Q\in \Omega_2} F_Q(x)^{p^r}\right).
\end{align}

We have the following characterization of 
Galois LCD codes.
\begin{theorem}\label{thm:lcd-rr}
A cyclic code $C=\langle g(x)\rangle$ of length $n=n_0p^r$ over $\F_{q^k}$ is Galois LCD if and only if
\begin{align}\label{gen-LCD}
g(x)=
\left(\prod_{Q\in J_1} f_Q(x)^{p^r}\right)
\left(\prod_{Q\in J_2} F_Q(x)^{p^r}\right)
\end{align}
for some subsets $J_1\subseteq \Omega_1$ and $J_2\subseteq \Omega_2$. Consequently, the number of Galois LCD cyclic codes of length $n$ over $\F_{q^k}$ is
\[
2^{|\Omega_1|+|\Omega_2|}.
\]
\end{theorem}

\begin{proof}
Write
\[
g(x)=
\left(\prod_{Q\in \Omega_1} f_Q(x)^{a_Q}\right)
\left(\prod_{Q\in \Omega_2}\ \prod_{R\in \mathcal O(Q)} f_R(x)^{a_R}\right),
\]
where $0\le a_Q,a_R\le p^r$.
Then
\[
h(x)=\frac{x^n-1}{g(x)}
=
\left(\prod_{Q\in \Omega_1} f_Q(x)^{p^r-a_Q}\right)
\left(\prod_{Q\in \Omega_2}\ \prod_{R\in \mathcal O(Q)} f_R(x)^{p^r-a_R}\right).
\]
Since each $Q\in \Omega_1$ is fixed by $\sigma$,  by Theorem~\ref{thm:reciprocal-coset}
\[
h^{\ast,\theta}(x)=
\left(\prod_{Q\in \Omega_1} f_Q(x)^{p^r-a_Q}\right)
\left(\prod_{Q\in \Omega_2}\ \prod_{R\in \mathcal O(Q)} f_R(x)^{p^r-a_{\sigma^{-1}(R)}}\right).
\]
By \eqref{eq-LCD}, the code $C$ is $\theta$-LCD if and only if
$
\gcd\bigl(g(x),h^{\ast,\theta}(x)\bigr)=1$.
Since the irreducible factors $f_R(x)$ are pairwise coprime, this holds if and only if
$
\min\{a_R,\;p^r-a_{\sigma^{-1}(R)}\}=0$ for every $R$.
Equivalently,
\[
a_R=0
\quad\text{or}\quad
a_{\sigma^{-1}(R)}=p^r
\]
for every $R$.

We now consider the two parts separately.
First, let $Q\in \Omega_1$. Since $\sigma(Q)=Q$, the above condition becomes
$
a_Q=0 \text{ or }
a_Q=p^r$.
Hence, each factor $f_Q(x)$ with $Q\in \Omega_1$ occurs either with multiplicity $0$ or with multiplicity $p^r$.

Next, let $Q\in \Omega_2$, and write
$
\mathcal O(Q)=\{Q_0,Q_1,\dots,Q_{m-1}\}
$,  $ \sigma(Q_i)=Q_{i+1}$
 where the indices are computed modulo $m$.
Then the condition above becomes
$
a_{Q_i}=0 $ or $
a_{Q_{i-1}}=p^r$ 
 for all $i\in \{0,1,\dots, m-1\}$.
If all exponents on the orbit are zero, then the orbit contributes nothing to $g(x)$. Suppose now that at least one exponent is $
a_{Q_i}>0$
for some $i$. Then necessarily
$
a_{Q_{i-1}}=p^r>0$. Using the same argument, we conclude that $
a_{Q_{i-2}}=p^r$.
Repeating this process around the orbit, we obtain
\[
a_{Q_0}=a_{Q_1}=\cdots=a_{Q_{m-1}}=p^r.
\]
Therefore, for each orbit represented by $Q\in \Omega_2$, either all exponents are $0$ or all exponents are $p^r$. Hence, the contribution of such an orbit is either $1$ or
$
F_Q(x)^{p^r}$, where 
$F_Q(x)=\prod_{R\in \mathcal O(Q)} f_R(x)$.
It follows that
\[
g(x)=
\left(\prod_{Q\in J_1} f_Q(x)^{p^r}\right)
\left(\prod_{Q\in J_2} F_Q(x)^{p^r}\right)
\]
for some subsets $J_1\subseteq \Omega_1$ and $J_2\subseteq \Omega_2$.

Conversely, assume that
\[
g(x)=
\left(\prod_{Q\in J_1} f_Q(x)^{p^r}\right)
\left(\prod_{Q\in J_2} F_Q(x)^{p^r}\right)
\]
for some subsets $J_1\subseteq \Omega_1$ and $J_2\subseteq \Omega_2$. Then
\[
h(x)=
\left(\prod_{Q\in \Omega_1\setminus J_1} f_Q(x)^{p^r}\right)
\left(\prod_{Q\in \Omega_2\setminus J_2} F_Q(x)^{p^r}\right).
\]
Since each $f_Q(x)$ with $Q\in\Omega_1$ and each $F_Q(x)$ with $Q\in\Omega_2$ is $\theta$-self-reciprocal, we have
\[
h^{\ast,\theta}(x)=
\left(\prod_{Q\in \Omega_1\setminus J_1} f_Q(x)^{p^r}\right)
\left(\prod_{Q\in \Omega_2\setminus J_2} F_Q(x)^{p^r}\right).
\]
Thus, $g(x)$ and $h^{\ast,\theta}(x)$ have no common irreducible factor, which implies that $
\gcd\bigl(g(x),h^{\ast,\theta}(x)\bigr)=1$.
By \eqref{eq-LCD}, the code $C$ is $\theta$-LCD.

Finally, each factor indexed by $\Omega_1$ and each orbit product indexed by $\Omega_2$ can be chosen independently. Hence, the total number of $\theta$-LCD cyclic codes is
\[
2^{|\Omega_1|+|\Omega_2|}
\]as desired.
\end{proof}
 
From Theorem \ref{thm:lcd-rr}, it  is equivalent to, each $\theta$-self-reciprocal irreducible factor $f_Q(x)$ with $Q\in\Omega_1$, and each orbit product $F_Q(x)$ with $Q\in\Omega_2$, occurs either with multiplicity $0$ or with multiplicity $p^r$.

 \begin{example}
Let $q=2$, $k=4$, and $T=1$. Then
$
\theta(a)=a^{q^T}=a^2$ for all $a\in \F_{16}$.
By Example~\ref{ex:x17minus1-F16},
\[
x^{17}-1
=
(x-1)\,A(x)\,B(x),
\]
where
$
A(x)=f_1(x)f_2(x)f_4(x)f_8(x)$ and $
B(x)=f_3(x)f_5(x)f_6(x)f_7(x)$.
Hence,
\[
x^{34}-1=(x^{17}-1)^2=(x-1)^2A(x)^2B(x)^2.
\]

For $T=1$, the $\sigma_1$-orbits are
$
\{Q_0\}$, $\{Q_1,Q_2,Q_4,Q_8\}$, and $\{Q_3,Q_5,Q_6,Q_7\}$ which implies that
$
\Omega_1=\{Q_0\} $ and $
\Omega_2=\{Q_1,Q_3\}$.
Therefore, by Theorem~\ref{thm:lcd-rr}, a cyclic code of length $34$ over $\F_{16}$ is $\theta$-LCD if and only if its generator polynomial has the form
\[
g(x)=(x-1)^{\varepsilon_0}A(x)^{\varepsilon_1}B(x)^{\varepsilon_2},
\]
where $\varepsilon_0,\varepsilon_1,\varepsilon_2\in\{0,2\}$.

For instance, the cyclic codes of length $34$ over $\F_{16}$ generated by
\[
(x-1)^2\bigl(f_1(x)f_2(x)f_4(x)f_8(x)\bigr)^2 \text{ and  } \bigl(f_1(x)f_2(x)f_4(x)f_8(x)\bigr)^2 \bigl(f_3(x)f_5(x)f_6(x)f_7(x)\bigr)^2\]
are  $\theta$-LCD.

Moreover,  the number of $\theta$-LCD cyclic codes of length $34$ over $\F_{16}$ is
\[
2^{|\Omega_1|+|\Omega_2|}=2^{1+2}=8
\] since $|\Omega_1|=1$ and $|\Omega_2|=2$.
\end{example}

 \subsection{Galois self-dual  cyclic codes}

For integers $0\le T<k$, let $
\theta(a)=a^{q^T} $
be the corresponding Galois automorphism of $\F_{q^k}$. Recall that a cyclic code $C$ is  {Galois self-dual}, or simply {$\theta$-self-dual} if $
C=C^{\perp_\theta}$.

  Let
$
n=n_0p^r$ 
with $r\ge 0$ and $\gcd(n_0,p)=1$. The following lemma shows that $\theta$-self-dual cyclic codes can occur only in the repeated-root case over fields of even characteristic.

\begin{lemma}\label{lem:selfdual-even}
There exists a $\theta$-self-dual cyclic code of length $n$ over $\F_{q^k}$ if and only if both $q$ and $n$ are even. In particular, simple root $\theta$-self-dual cyclic codes do not exist.
\end{lemma}

\begin{proof}
Assume that  there exists a $\theta$-self-dual cyclic code $
C=\langle g(x)\rangle$ 
of length $n$ over $\F_{q^k}$. Let
$
h(x)=\frac{x^n-1}{g(x)}$.
By \eqref{eq-SD}, we have 
$
g(x)=h^{\ast,\theta}(x)$. Since  the factor $x-1$ is always $\theta$-self-reciprocal,  the exponent of $x-1$ in $g(x)$ must be equal to the exponent of $x-1$ in $h^{\ast,\theta}(x)$ which is 
$
\frac{p^r}{2}.
$
Therefore, $p^r$ must be even. Preciselt,  $p=2$ and $r\geq 1$. Since $q$ is a power of $p=2$, it follows that $q$ is even.  Therefore, $n=n_0p^r$ is even. In particular, simple root $\theta$-self-dual cyclic codes do not exist.

Conversely, assume that both $q$ and $n$ are even. Then  $\F_{q^k}$ has characteristic $2$. Since $n$ is even, the polynomial
\[
g(x)=x^{n/2}-1
\]
is  a well-defined monic divisor of  $x^n-1$.  
Since  the characteristic is $2$, we have
\[
x^n-1=x^n+1=(x^{n/2}+1)^2=(x^{n/2}-1)^2
\]
and 
\[
h(x)=\frac{x^n-1}{g(x)}=x^{n/2}-1=g(x).
\]
Sine $\theta$ acts trivially on the subfield $\F_q$,
\[
h^{\ast,\theta}(x)=g^{\ast,\theta}(x)=g(x).
\] 
By \eqref{eq-SD}, the cyclic code $C=\langle g(x)$  of length $n$ over $\F_{q^k}$ is $\theta$-self-dual. 
\end{proof}

We now restrict the  attention to repeated-root cyclic codes. By Lemma~\ref{lem:selfdual-even}, $\theta$-self-dual cyclic codes can exist only when $q$ is even. Hence, throughout this subsection, we write $
q=2^\nu $ and $
n=n_02^r$,
where $\nu\ge1$, $r\ge 1$ and $n_0$ is odd. 
The next theorem gives a      characterization for the generator polynomial of $\theta$-self-dual  cyclic codes.

\begin{theorem}\label{thm:rr-selfdual-general} Let $
q=2^\nu $ and $
n=n_02^r$ for some  $\nu\ge 1$, $r\ge 1$ and $n_0$ is odd. 
Let  $g(x)$ be a monic divisor of $x^n-1$ over  in \eqref{xn-1-expan} of the form
\[
g(x)=
\left(\prod_{Q\in \Omega_1} f_Q(x)^{a_Q}\right)
\left(\prod_{Q\in \Omega_2}\ \prod_{R\in \mathcal O(Q)} f_R(x)^{a_R}\right),
\]
where $
0\le a_Q,a_R\le 2^r$.
Then the cyclic code
$
C=\langle g(x)\rangle$  of length  $n$ over $\F_{q^k}$
is $\theta$-self-dual if and only if
\[
a_Q=2^{r-1},
\]
for every $Q\in \Omega_1$, and
\[
a_R=2^r-a_{\sigma^{-1}(R)}
\]
for every $R$ in 
  a $\sigma$-orbit represented by some element of $\Omega_2$.
\end{theorem}

\begin{proof}
Let
\[
h(x)=\frac{x^n-1}{g(x)}.
\]
Then
\[
h(x)=
\left(\prod_{Q\in \Omega_1} f_Q(x)^{2^r-a_Q}\right)
\left(\prod_{Q\in \Omega_2}\ \prod_{R\in \mathcal O(Q)} f_R(x)^{2^r-a_R}\right).
\]
By Theorem~\ref{thm:reciprocal-coset},
\[
h^{\ast,\theta}(x)=
\left(\prod_{Q\in \Omega_1} f_Q(x)^{2^r-a_Q}\right)
\left(\prod_{Q\in \Omega_2}\ \prod_{R\in \mathcal O(Q)} f_R(x)^{2^r-a_{\sigma^{-1}(R)}}\right),
\]
since every $Q\in\Omega_1$ is fixed by $\sigma$.

First,  assume first that $C$ is $\theta$-self-dual. By \eqref{eq-SD},
\[
g(x)=h^{\ast,\theta}(x).
\]
Comparing exponents of the pairwise coprime irreducible factors, we obtain
$
a_Q=2^r-a_Q$ for every $Q\in\Omega_1$, and 
$
a_R=2^r-a_{\sigma^{-1}(R)}$ 
 for every $R$
in  a $\sigma$-orbit represented by some element of $\Omega_2$. From the first equality, we have 
$ 
a_Q=2^{r-1}$ for every $ Q\in\Omega_1.
$

Conversely, assume that
$
a_Q=2^{r-1} $ for every $Q\in\Omega_1$,
and
$
a_R=2^r-a_{\sigma^{-1}(R)}$ 
 for every $R$
in  a $\sigma$-orbit represented by some element of $\Omega_2$. Then, substituting these equalities into the expression for $h^{\ast,\theta}(x)$, we obtain
\[
h^{\ast,\theta}(x)=
\left(\prod_{Q\in \Omega_1} f_Q(x)^{a_Q}\right)
\left(\prod_{Q\in \Omega_2}\ \prod_{R\in \mathcal O(Q)} f_R(x)^{a_R}\right)
=
g(x).
\]
Hence, by \eqref{eq-SD}, the cyclic code $C$ is $\theta$-self-dual.
\end{proof}

\begin{remark}\label{rem:orbit-parity}
Let $Q\in\Omega_2$ and write
$
\mathcal O(Q)=\{Q_0,Q_1,\dots,Q_{m-1}\}
$
with
$
\sigma(Q_i)=Q_{i+1}
$, where the subscripts are computed modulo $m$.
Then Theorem~\ref{thm:rr-selfdual-general} is equivalent to
\[
a_{Q_i}=2^r-a_{Q_{i-1}}
\]  {for all } $i\in \{0,1,\dots, m-1\}$. 
Thus, if $m$ is even, the exponents alternate between two complementary values
\[
b,\ 2^r-b,\ b,\ 2^r-b,\dots, b,\ 2^r-b.
\]
If $m$ is odd, then going once around the orbit forces
\[
a_{Q_i}=2^{r-1} 
\]{for all } $i\in \{0,1,\dots, m-1\}$.
\end{remark}

The following corollary makes the generator polynomial explicit in the involutory case.

\begin{corollary}\label{cor:rr-selfdual-involutory} Let  $
q=2^\nu $ for some $nu\ge 1$ and let $k$ and $T$ be integers such that $0\leq T<k$.
Assume that the map
$
Q\mapsto -q^TQ
$
is an involution on the set of $q^k$-cyclotomic classes modulo $n_0$. Then every orbit represented by an element of $\Omega_2$ has size $2$. A repeated-root cyclic code $
C=\langle g(x)\rangle$ 
of length $n=n_02^r$ over $\F_{q^k}$ is Galois self-dual if and only if
\[
g(x)=
\left(\prod_{Q\in \Omega_1} f_Q(x)^{2^{r-1}}\right)
\left(\prod_{Q\in \Omega_2} f_Q(x)^{b_Q}f_{\sigma(Q)}(x)^{2^r-b_Q}\right),
\] where $0\le b_Q\le 2^r$.
Consequently, the number of Galois self-dual repeated-root cyclic codes of length $n$ over $\F_{q^k}$ is
$
(2^r+1)^{|\Omega_2|}$.
\end{corollary}

\begin{proof} Since  $
Q\mapsto -q^TQ
$
is an involution, it is clear that 
 every orbit represented by an element of $\Omega_2$ has size $2$. The general condition in Theorem~\ref{thm:rr-selfdual-general} reduces to $
a_Q=2^{r-1} $ for all $Q\in\Omega_1$,
and $
a_{\sigma(Q)}=2^r-a_Q
 $ for all $Q\in\Omega_2$.
Writing $b_Q=a_Q$ gives the asserted form of the generator polynomial.

Conversely, any polynomial of the stated form satisfies
$
g(x)=h^{\ast,\theta}(x),
$
and hence, the corresponding code is Galois self-dual by \eqref{eq-SD}.

Finally, each orbit represented by an element of $\Omega_2$ contributes exactly $2^r+1$ choices for the exponent $b_Q$. Therefore, the total number of such codes is $
(2^r+1)^{|\Omega_2|}$.
\end{proof}

\begin{example}\label{ex:theta-selfdual-34-T1}
Let $q=2$, $k=4$, and $T=1$. Then $
\theta(a)=a^{q^T}=a^2
$ for all $ a\in \F_{16}$.
By Example~\ref{ex:x17minus1-F16}, the $\sigma_1$-orbits on the $2^4$-cyclotomic classes modulo $17$ are
$
\{Q_0\}, \{Q_1,Q_2,Q_4,Q_8\}$, and $ \{Q_3,Q_5,Q_6,Q_7\}$.
Hence, $
\Omega_1=\{Q_0\}$ and $
\Omega_2=\{Q_1,Q_3\}$. It follows that 
\[
x^{17}-1=(x-1)\,A(x)\,B(x),
\]
where
$
A(x)=f_1(x)f_2(x)f_4(x)f_8(x)$ and $ 
B(x)=f_3(x)f_5(x)f_6(x)f_7(x)$ defined in Example~\ref{ex:x17minus1-F16}.
Therefore,
\[
x^{34}-1=(x^{17}-1)^2=(x-1)^2A(x)^2B(x)^2.
\]
Since the orbit sizes in $\Omega_2$ are $4$, Corollary~\ref{cor:rr-selfdual-involutory} does not apply. Instead, by Theorem~\ref{thm:rr-selfdual-general},
\[
a_{Q_0}=1,
\]
and on each orbit of length $4$, the exponents satisfy
$
a_{Q_i}=2-a_{Q_{i-1}}$.
For example, the choices
$
(a_{Q_1},a_{Q_2},a_{Q_4},a_{Q_8})=(0,2,0,2)
$
and
$
(a_{Q_3},a_{Q_6},a_{Q_5},a_{Q_7})=(1,1,1,1)
$
satisfy this condition. Hence
\[
g(x)=(x-1)\,f_2(x)^2f_8(x)^2f_3(x)f_5(x)f_6(x)f_7(x)
\]
generates a $\theta$-self-dual cyclic code of length $34$ over $\F_{16}$.
\end{example}

\begin{example}\label{ex:theta-selfdual-34-T2}
Let $q=2$, $k=4$, and $T=2$. Then
$
\theta(a)=a^{q^T}=a^4$
for all $a\in \F_{16}$.
By Example~\ref{ex:x17minus1-F16}, the $\sigma_2$-orbits are
$
\{Q_0\}$, $\{Q_1,Q_4\}$, $ \{Q_2,Q_8\}$ $ \{Q_3,Q_5\}$,  and $\{Q_6,Q_7\}$.
Thus, the involutory case applies, since each nontrivial orbit has size $2$. From Example~\ref{ex:x17minus1-F16}, we have 
\[
x^{34}-1=(x-1)^2\prod_{i=1}^8 f_i(x)^2.
\]
By Corollary~\ref{cor:rr-selfdual-involutory}, every Galois self-dual cyclic code of length $34$ over $\F_{16}$ has generator polynomial of the form
\[
g(x)
=
(x-1)
f_1(x)^{b_1}f_4(x)^{2-b_1}
f_2(x)^{b_2}f_8(x)^{2-b_2}
f_3(x)^{b_3}f_5(x)^{2-b_3}
f_6(x)^{b_4}f_7(x)^{2-b_4},
\]
where
$
0\le b_1,b_2,b_3,b_4\le 2$.
For example, choosing
$
b_1=b_2=b_3=b_4=1$,
it follows that the Galois self-dual cyclic code of length $34$ over $\F_{16}$ generated by
\[
g(x)=(x-1)\,f_1(x)f_4(x)f_2(x)f_8(x)f_3(x)f_5(x)f_6(x)f_7(x)=x^{17}-1.
\]
Moreover, the number of $\theta$-self-dual cyclic codes of length $34$ over $\F_{16}$ is
$
(2^1+1)^4=3^4=81$.
\end{example}

\section{Conclusion}
\label{sec6}

In this paper, a new family of good integers, namely the $(T,k)$-good integers with respect to $(a,b)$, has been introduced and studied. This family extends the classical notions of good, oddly-good, and evenly-good integers, and provides a   arithmetic framework for divisibility conditions of the form
\[
d\mid \bigl(a^{ks+T}+b^{ks+T}\bigr).
\]
A self-contained arithmetic theory has been developed, including local criteria at odd prime powers, a global characterization of odd integers in terms of $2$-adic valuations of multiplicative orders, a treatment of the even case, and an explicit algorithm for deciding membership in $G_{(T,k)}(a,b)$ and constructing a corresponding exponent.

The specialization $(a,b)=(q,1)$ shows that this arithmetic is closely connected with the structure of $q^k$-cyclotomic classes. In particular, the membership condition
$
d\in G_{(T,k)}(q,1)$
determines when a $q^k$-cyclotomic class is stable under the action $
a\mapsto -q^Ta$.
This orbit  interpretation leads naturally to applications in coding theory. In the cyclic case,  we present a characterization of $\theta$-self-reciprocal irreducible factors of $x^n-1$ over $\F_{q^k}$, and consequently explicit generator-polynomial descriptions and enumeration formulas for simple root $\theta$-self-dual and $\theta$-LCD cyclic codes. These applications  have been developed at the level of general theory. It would be of further interest to focus on specific code lengths and classes of codes for which the arithmetic properties of $G_{(T,k)}(q,1)$ give more explicit structural descriptions and sharper enumeration formulas.

The abelian group formulation developed in this paper indicates that the same arithmetic framework extends beyond cyclic codes. This suggests a natural direction for further work on  abelian codes in group algebras $\F_{q^k}[G]$, where $q^k$-cyclotomic classes in finite abelian groups play the role of the corresponding cyclotomic classes modulo $n$. It would also be of interest to investigate analogous questions for repeated-root settings, constacyclic and quasi-abelian codes, and other families of Galois dual codes.

\bmhead{Acknowledgements}
This research was supported by the National Research Council of Thailand and Silpakorn University under Research Grant N42A650381.

\section*{Declarations}

\begin{itemize}
\item Funding:  This research was supported by the National Research Council of Thailand and Silpakorn University under Research Grant N42A650381.
\item Conflict of interest/Competing interests:  The authors declare no potential conflict of interest. 
\item Ethics approval and consent to participate: Not applicable.
\item Consent for publication: The  authors have read and approved the final manuscript and consent to its publication.
\item Data availability: Data sharing is not applicable to this article as no datasets were generated or analyzed during the current study.
\item Materials availability:  Not applicable.
\item Code availability:  Not applicable.
\item Author contribution:  S.~Jitman was responsible for the conceptualization, methodology, formal analysis, proofs, algorithm design, writing of the manuscript, and partial investigation in Section~4. P.~Boonsuriyatham contributed to the investigation in Section~4. Both authors verified the final version of the manuscript and approved it for publication.
 
\end{itemize}

\end{document}